\documentclass[12pt, leqno]{article}
\usepackage{amsmath}
\usepackage{amsthm}
\usepackage{amssymb}
\usepackage{latexsym}
\usepackage{graphicx}
\allowdisplaybreaks[1]
\usepackage{amsfonts}
\usepackage{ascmac}
\usepackage{color}
\usepackage{enumerate}

\theoremstyle{definition}
\theoremstyle{plain}
\allowdisplaybreaks[1]
\date{}
\newtheorem{Thm}{Theorem}[section]
\newtheorem{Prop}[Thm]{Proposition}
\newtheorem{Lemma}[Thm]{Lemma}

\newtheorem{Rem}[Thm]{Remark}

\newcommand{\dt}{\Delta t}
\newcommand{\dx}{\Delta x}

\newcommand{\p}{\partial}

\newcommand{\dis}{\displaystyle}

\newcommand{\norm}{\parallel}

\newcommand{\Z}{{\mathbb Z}}

\newcommand{\N}{{\mathbb N}}
\newcommand{\R}{{\mathbb R}}

\newcommand{\G}{ \mathcal{G}}
\newcommand{\tG}{ \tilde{\mathcal{G}}}
\newcommand{\1}{{\bf 1} }

\newcommand{\ep}{\varepsilon }

\def\text#1{\mbox{#1 }}

\title{\bf Stochastic and variational approach to finite difference approximation of Hamilton-Jacobi equations}
\author{Kohei Soga
\footnote{Department of Mathematics, Faculty of Science and Technology, Keio University, 3-14-1 Hiyoshi, Kohoku-ku, Yokohama, 223-8522, Japan. E-mail:  soga@math.keio.ac.jp 
}}
\begin{document}
\maketitle
\begin{abstract} 
\noindent The author presented a stochastic and variational approach to the Lax-Friedrichs finite difference scheme applied to hyperbolic scalar conservation laws and the corresponding Hamilton-Jacobi equations with convex and superlinear Hamiltonians  in the one-dimensional periodic setting, showing new results on the stability and convergence of the scheme [Soga, Math. Comp. (2015)]. In the current paper, we extend these results to the higher dimensional setting. Our framework  with a deterministic scheme provides approximation of  viscosity solutions of Hamilton-Jacobi equations, their spatial derivatives and the backward characteristic curves at the same time, within an arbitrary time interval.  The proof is based on stochastic calculus of variations with random walks; a priori boundedness of minimizers of the variational problems that verifies a CFL type stability condition; the law of large numbers for random walks under the hyperbolic scaling limit.  Convergence of approximation and the rate of convergence are obtained in terms of probability theory.   The idea is reminiscent of the stochastic and variational approach to the vanishing viscosity method introduced in [Fleming, J. Differ. Eqs (1969)]. 
 
\medskip

  \medskip

\end{abstract}
%
\setcounter{section}{0}
\setcounter{equation}{0}
\section{Introduction}
We consider finite difference approximation to viscosity solutions of initial value problems for Hamilton-Jacobi equations
\begin{eqnarray}\label{HJ}
\left\{
\begin{array}{lll}
&v_t(x,t)+H(x,t,v_x(x,t))=h\mbox{\quad in $\R^d\times(0,T]$},
\medskip\\
&v(x,0)=v^0(x)\mbox{\quad on $\R^d$,\quad $v^0\in Lip_r(\R^d)$, $|v^0|\le R$},
\end{array}
\right.
\end{eqnarray}
where $v_t:=\frac{\partial v}{\partial t}$, $v_x:=(\frac{\partial v}{\partial x^1},\ldots,\frac{\partial v}{\partial x^d})$, $d\ge1$, $h$ is a given constant and $Lip_r(\R^d)$ denotes the family of Lipschitz functions: $\R^d\to\R$ with Lipschitz constants bounded by $r>0$. Initial data is assumed to be bounded by $R>0$. We arbitrarily fix the constants $T,r,R$. The function $H$ is assumed to satisfy the following (H1)--(H5):
\begin{enumerate}
\item[(H1)] $H(x,t,p):\R^d\times\R\times\R^d\to\R$, $C^2$,
\item[(H2)] $H_{pp}(x,t,p)$ is positive definite in $\R^d\times\R\times\R^d$,
\item[(H3)] $H$ is uniformly superlinear with respect to $p$, namely, for each $a\ge0$ there exists $b_1(a)\in\R$ such that $H(x,t,p)\ge a\norm p\norm+b_1(a)$  in $\R^d\times\R\times\R^d$,
\item[(H4)] $H, H_{x^i},H_{p^i}, H_{x^ix^j},H_{x^ip^j},H_{p^ip^j}$ are uniformly bounded on $\R^d\times\R\times K$ for each compact set $K\subset \R^d$  for $i,j=1,\ldots,d$,
\item[(H5)] For the Legendre transform of $H(x,t,\cdot)$, denoted by $L$, there exists $\alpha>0$ such that $| L_{x^j}|\le \alpha (1+|L|)$ in $\R^d\times\R\times\R^d$ for $j=1,\ldots,d$.
\end{enumerate} 
Here, $\norm x\norm:=\sqrt{\sum_{1\le j\le d}(x^j)^2}$ and $x\cdot y:=\sum_{1\le j\le d} x^jy^j$ for $x,y\in\R^d$. Note that, due to (H1)--(H3), the function $L(x,t,\xi):\R^d\times\R\times\R^d\to\R$ is well-defined and is given by
$$L(x,t,\xi)=\sup_{p\in\R^d}\{p\cdot\xi-H(x,t,p)\}.$$ 
We will see properties of $L$ in Section 3. 

These problems arise in many fields such as theories of optimal control, dynamical systems and so on. Our motivation comes from  weak KAM theory, which connects viscosity solutions of Hamilton-Jacobi equations and Hamiltonian/Lagrangian dynamics. 
The central objects in weak KAM theory are viscosity solutions, their spatial derivatives and the characteristic curves.  In numerical analysis of weak KAM theory, it is important to develop a method that is able to approximate all of these objects at the same time. See \cite{Nishida-Soga}, \cite{Soga2}, \cite{Soga3}, \cite{Soga4} for recent development of numerical analysis of weak KAM theory based on such a method in one-dimensional problems with Tonelli Hamiltonians, i.e., $H$ is periodic in $(x,t)$ with (H1)--(H3). In the one-dimensional case,  (\ref{HJ}) is equivalent to the scalar conservation law  
\begin{eqnarray}\label{CL}
\left\{
\begin{array}{lll}
&u_t+H(x,t,u)_x=0\mbox{\quad in $\R\times(0,T]$},
\medskip\\
&u(x,0)=u^0(x)\mbox{\quad on $\R$,\quad $u^0\in L^\infty(\R)$}.
\end{array}
\right.
\end{eqnarray}
If $u^0=v^0_x$, the viscosity solution $v$ or entropy solution $u$ is derived from the other and they satisfy the relation $u=v_x$. Therefore, approximation of $u$ implies approximation of $v$. However, the other way around is not necessarily true. In the pioneering work \cite{Oleinik} on finite difference approximation of (\ref{CL}) with a wide class of functions $H$, stability and convergence of the Lax-Friedrichs finite difference scheme are proved within a restricted time interval based on a functional analytic approach. The restriction depends on the growth rate of $H$ for $|p|\to\infty$. 
The author recently announced a stochastic and variational approach to the Lax-Friedrichs scheme \cite{Soga2}, where the discretized equation of (\ref{CL}) with  the Lax-Friedrichs scheme is converted to a discretized equation of (\ref{HJ}), and stability and convergence are proved within an arbitrary time interval. Furthermore, this framework guarantees approximation of all of entropy solutions, viscosity solutions and their characteristic curves at the same time. Application of these results to weak KAM theory is found in  \cite{Soga3}, \cite{Soga4}.  
The key point of the stochastic and variational approach is to characterize the so-called {\it numerical viscosity} of the discretized equation by space-time inhomogeneous random walks. Then, we obtain a stochastic Lax-Oleinik type operator with random walks for the discretized Hamilton-Jacobi equation. Convergence of the stochastic Lax-Oleinik type operator to the exact one for (\ref{HJ}) is proved through the law of large numbers for  random walks. The idea is reminiscent of the stochastic and variational approach to the vanishing viscosity method with the Brownian motions \cite{Fleming}. 

The purpose of this paper is to generalize the results in \cite{Soga2} to the higher dimensional problems. Here, we do not restrict ourselves to the case with Tonelli Hamiltonians, also because such a non-compact problem arises in different contexts.  
We introduce a simple (deterministic) scheme  that is a direct generalization of the one-dimensional Lax-Friedrichs scheme on a staggered grid. Then, we formulate a stochastic Lax-Oleinik type operator with space-time inhomogeneous random walks on a grid in $\R^d$.  Convergence of approximation is proved through the law of large numbers for random walks. {\it We will have convergent approximation of the viscosity solutions, their spatial derivatives and the characteristic curves at the same time within each arbitrarily time interval}. In the case of $d\ge 2$, there is no equivalence between Hamilton-Jacobi equations and scaler conservation laws. Therefore, we discretize the Hamilton-Jacobi equation so that  both of difference viscosity solutions and their  discrete spatial derivatives can be controllable. 
Our approach yields new results on stability of the scheme and its convergence to exact viscosity solutions with the rate $O(\sqrt{\dx})$, from which convergent approximation of  derivatives and characteristic curves is derived. Here, $\dx,\dt>0$ are spatial, temporal discretization parameters respectively. In our proof, this rate of convergence comes from the hyperbolic scaling limit (i.e., $\dx,\dt\to0$ with $\dt/\dx=O(1)$) of random walks.   The result of the present paper would be a basic tool for numerical analysis of Hamilton-Jacobi equations including weak KAM theory in the higher dimensional setting.  

 There are many results on finite difference approximation of the viscosity solution to (\ref{HJ}). We refer to the pioneering work \cite{Crandall-Lions} and its generalization \cite{Souganidis}, where convergence of a class of finite difference schemes with the rate $O(\sqrt{\dx})$ is proved in an abstract setting, under the assumption that schemes are monotone (verification of this assumption is necessary for each scheme). In the recent work \cite{BFZ}, a numerical scheme for Hamilton-Jacobi equations with separable Hamiltonians, i.e., functions $H$ of  the form $H(x,t,p)=f(p)+g(x,t)$, is developed based on a direct discretization of Lax-Oleinik type operators.  As far as the author knows in literature,  the functions $H$ with (H1)--(H5) are not covered; furthermore, convergence of approximation is proved only for viscosity solutions. The main difficulty of finite difference approximation in (\ref{HJ}) with our requirement  is to verify a priori boundedness of the discrete derivatives of difference solutions (this yields  monotonicity of the scheme). Verification of  the a priori boundedness is harder even in  one-dimensional cases, if $H(x,t,p)$ is convex and superlinear with respect to $p$ and is not of a separable form like $H(x,t,p)=f(p )+g(x,t)$.  
This difficulty is overcome by means of our stochastic and variational framework, where we may effectively use a priori boundedness of minimizers of the stochastic  Lax-Oleinik type operator.  In our framework, hyperbolic scaling $\dt/\dx=O(1)$ is crucial. It is interesting to  remark that  the scheme developed in \cite{BFZ} accepts $\dx,\dt\to0$ with $\dx/\dt\to0$, say $\dt=\sqrt{\dx}$, which is not hyperbolic scaling nor diffusive scaling. If one only needs approximation of viscosity solutions with a separable Hamiltonian, the scheme in \cite{BFZ} would be quicker than ours in actual computation. 
\medskip

\noindent {\bf Acknowledgement.}  The main part of this work was done, when the author belonged to Unit\'e de math\'ematiques pures et appliqu\'ees, CNRS UMR 5669  \&  \'Ecole Normale Sup\'erieure de Lyon, being supported by ANR-12-BS01-0020 WKBHJ as a researcher for academic year 2014-2015, hosted by Albert Fathi. 
The author is partially supported by JSPS Grant-in-aid for Young Scientists (B) \#15K21369.
\setcounter{section}{1}
\setcounter{equation}{0}
\section{Scheme and result}

Let $\dx>0$ and $\dt>0$ be discretization parameters for space and time respectively. Set $\Delta:=(\dx,\dt)$, $|\Delta|:=\max\{\dx,\dt \}$, $G_{even}:=\{m\dx \,|\,m=(m^1,\ldots,m^d)\in\Z^d,m^1+\cdots+ m^d=even\}$,  $G_{odd}:=\{m\dx \,|\,m\in\Z^d,m^1+\cdots+m^d=odd\}$, $t_k:=k\dt$ for $k\in \N\cup\{0\}$ and
\begin{eqnarray*} 
&&\mathcal{G}:=\bigcup_{k\ge0}\Big\{(G_{even}\times\{t_{2k}\})\cup(G_{odd}^d\times\{t_{2k+1}\})\Big\},\\
&&\tG:=\bigcup_{k\ge0}\Big\{(G_{odd}\times\{t_{2k}\})\cup(G_{even}\times\{t_{2k+1}\})\Big\}.
\end{eqnarray*}
Note that $\G\cup\tG=(\dx\Z)^d\times(\dt\Z_{\ge0})$, where $\dx \Z:=\{i\dx\,|\,i\in\Z\}$. Each point of $(\dx\Z)^d\times(\dt\Z_{\ge0})$ is denoted by $(x_m,t_k)=(x^1_{m^1},\ldots,x^d_{m^d},t_k)$ with $m=(m^1,\ldots,m^d)\in\Z^d$. We sometimes use the notation $(x_{m},t_k),(x_{m+\1},t_{k+1})$ to indicate points of $\G$ and $(x_{m+\1},t_k),(x_{m},t_{k+1})$ to indicate points of $\tG$ with $\1:=(1,0,\ldots,0)\in\Z^d$.  For $(x,t)\in \G\cup \tG$, the notation $m(x)$, $k(t)$ denotes the index of $x$, $t$ respectively. For $t\ge0$, $k(t)$ denotes the integer such that $t\in[k(t)\dt,k(t)\dt+\dt)$.  
Let $\{e_1,\ldots,e_d\}$ be the standard basis of $\R^d$. Set 
$$B:=\{\pm e_1,\ldots,\pm e_d\}.$$

Let $v=v^k_{m+\1}$ denote a function: $\tG\ni(x_{m+\1},t_k)\mapsto v^k_{m+\1}\in\R$. Introduce the spatial difference derivatives of $v$ that are defined at each point $(x_m,t_k)\in\G$ as 
\begin{eqnarray*}
(D_{x^j}v)^k_{m}:=\frac{v^k_{m+e_j}-v^k_{m-e_j}}{2\dx},\quad
(D_{x}v)^k_{m}:=\big((D_{x^1}v)^k_{m},\ldots,(D_{x^d}v)^k_{m}\big).
\end{eqnarray*}
Introduce the temporal difference derivative of $v$ as 
$$(D_t v)^{k+1}_m:=\left( v^{k+1}_{m}-\frac{1}{2d}\sum_{\omega\in B}v^k_{m+\omega} \right)\frac{1}{\dt}.$$
Discretize  (\ref{HJ}) as 
\begin{eqnarray}\label{HJ-Delta}
\left\{
\begin{array}{lll}
&(D_tv)^{k+1}_m+H(x_m,t_k,(D_xv)^k_m)=h\mbox{\quad in $\tG|_{0\le k\le k(T)}$},
\medskip\\
&v^0_{m+\1}\mbox{\quad  is given on $G_{odd}$},
\end{array}
\right.
\end{eqnarray}
where the initial data $v^0_{m+\1}$ is one of the following:
\begin{eqnarray}\label{initial-1}
v^0_{m+\1}&:=&v^0(x_{m+\1}),\\\label{initial-2}
v^0_{m+\1}&:=&\frac{1}{(2\dx)^{d}} \int_{[-\dx,\dx]^d} v^0(x_{m+\1}+y)dy.
\end{eqnarray} 
Note that,  in (\ref{HJ-Delta})$, v^{k+1}_{m}$ is unknown determined by $v^k_{m+\1}$ as a recursion. If $d=1$,  the difference equation in (\ref{HJ-Delta}) is exactly the same as the scheme studied in \cite{Soga2}. To prove convergence of $(D_xv)^k_m$ without assuming semiconcavity of exact initial data $v^0$, we need \eqref{initial-2}.  If we do not refer to the form of discrete initial data in an assertion, both  \eqref{initial-2} and \eqref{initial-2} are fine.

We introduce space-time inhomogeneous random walks on $\tG$, which correspond to characteristic curves of (\ref{HJ}). 
For each point  $(x_n,t_{l+1})\in\tG$, we consider the backward random walks $\gamma$ within $[t_{l'},t_{l+1}]$ which start from $x_n$ at $t_{l+1}$ and move by $\omega\dx$, $\omega\in B$ in each backward time step $\dt$: 
$$\gamma=\{\gamma^k\}_{k=l',\cdots,l+1},\quad\gamma^{l+1}=x_{n},\quad \gamma^{k}=\gamma^{k+1}+\omega\dx.$$
More precisely, we set the following objects for each $(x_n,t_{l+1})\in\tG$ and $l'\le l$: 
\begin{eqnarray*}
&&X_n^{l+1,k}:=\{ x_{m+\1} \,|\, \mbox{ $(x_{m+\1},t_{k})\in\tG$, $\max_{1\le j\le d}|x^j_{m^j+\1^j}-x^j_{n^j}|\le(l+1-k)\dx$}\}\\
 &&\mbox{(the set of reachable points at time $k$)},\\
&&G_n^{l+1,l'}:=\bigcup_{l'\le k\le l+1}\big(X_n^{l+1,k}\times\{t_{k}\}\big)\subset\tG, \\
&&\xi:G_n^{l+1,l'+1}\ni(x_m,t_{k+1})\mapsto\xi^{k+1}_m\in[-(d\lambda)^{-1},(d\lambda)^{-1}]^d,\quad \lambda:=\dt/\dx, \\
&&\rho: G_n^{l+1,l'+1}\times B \ni(x_m,t_{k+1};\omega)\mapsto\rho^{k+1}_m(\omega):=\Big\{\frac{1}{2d}-\frac{\lambda}{2}(\omega\cdot\xi^{k+1}_m)\Big\}\in[0,1],\\
&&\gamma:\{ l',l'+1,\ldots,l+1\}\ni k\mapsto \gamma^k\in X_n^{l+1,k},\mbox{ $\gamma^{l+1}=x_n$, $\gamma^{k}=\gamma^{k+1}+\omega\dx$, $\omega\in B$},\\
&&\Omega_n^{l+1,l'}:\mbox{ the family of the above $\gamma$}, 
\end{eqnarray*}
where $\xi$ and $\rho$ are not defined at $l'$. We may regard $\rho^{k+1}_m(\omega)$, $\omega\in B$ as the transition probability from $(x_m,t_{k+1})$ to $(x_m+\omega\dx,t_k)$, since we have 
$$\sum_{\omega\in B}\rho^{k+1}_m(\omega)=\sum_{i=1}^d(\rho^{k+1}_m(e_i)+\rho^{k+1}_m(-e_i))=1.$$
We control transition of random walks by $\xi$, which is a kind of velocity field on $G_n^{l+1,l'}$.     
We define the density of each path $\gamma\in\Omega_n^{l+1,l'}$ as 
$$\mu_n^{l+1,l'}(\gamma):=\prod_{l'\le k\le l}\rho^{k+1}_{m(\gamma^{k+1})}(\omega^{k+1}),$$    
where $\omega^{k+1}:=(\gamma^k-\gamma^{k+1})/\dx$. For each $\xi$, the density $\mu_n^{l+1,l'}(\cdot)=\mu_n^{l+1,l'}(\cdot;\xi)$ yields a probability measure of $\Omega_n^{l+1,l'}$, namely, 
$$prob(A)=\sum_{\gamma\in A}\mu_n^{l+1,l'}(\gamma;\xi)\mbox{\quad for $A\subset\Omega_n^{l+1,l'}$}. $$
The expectation with respect to this probability measure is denoted by $E_{\mu_n^{l+1,l'}(\cdot;\xi)}[\cdot]$, namely, for a random variable $f:\Omega_n^{l+1,l'}\to\R$,
$$E_{\mu_n^{l+1,l'}(\cdot;\xi)}[f(\gamma)]:=\sum_{\gamma\in\Omega_n^{l+1,l'}}\mu_n^{l+1,l'}(\gamma;\xi)f(\gamma).$$
We remark that, since our transition probabilities are space-time inhomogeneous, the well-known law of large numbers and central limit theorem for random walks do not always hold. The author investigated the asymptotics  of the probability measure of $\Omega_n^{l+1,l'}$ as $\Delta\to0$ under hyperbolic scaling in the one-dimensional case \cite{Soga1}. We extend this investigation to the current multi-dimensional case.    

Let $v(x,t)$ be the viscosity solution of (\ref{HJ}) ($v$ uniquely exists as a Lipschitz function). Then, $v$ satisfies for each $x\in\R^d$ and $t>0$,
$$v(x,t)=\inf_{\gamma\in AC,\,\,\gamma(t)=x}\left\{ \int^t_0 L(\gamma(s),s,\gamma'(s))ds+v^0(\gamma(0)) \right\}+ht,
$$
where $AC$ is the family of absolutely continuous curves $\gamma:[0,t]\to\R^d$ (a Lax-Oleinik type operator for (\ref{HJ})). Due to Tonelli's theory, there exists a minimizing curve $\gamma^\ast$ for $v(x,t)$, which is a $C^2$-solution of the Euler-Lagrange equation generated by $L$. We say that a point $(x,t)\in\R^d\times\R$ is regular, if $v_x(x,t)$ exists. The viscosity solution $v$ of (\ref{HJ}) is Lipschitz and hence it is differentiable a.e. If $(x,t)$ is regular, the minimizing curve $\gamma^\ast$ for $v(x,t)$ is unique. Let $(x,t)$ be regular and let $\gamma^\ast$ be the minimizing curve for $v(x,t)$. Then, we have 
$$v_x(x,t)=\int^t_0L_x(\gamma^\ast(s),s,\gamma^\ast{}'(s))ds+v^0_x(\gamma^\ast(0)),$$ 
where $v^0$ is supposed to be semiconcave; otherwise, $v^0_x(\gamma^\ast(0))$ must be replaced by $L_\xi(\gamma^\ast(0),0,\gamma^\ast{}'(0))$. We refer to \cite{Cannarsa}, \cite{Fathi-book} for more on viscosity solutions and calculus of variations. 
 
Now we state the main results. The first theorem shows a stochastic and variational representation of the difference solution to (\ref{HJ-Delta}) (a Lax-Oleinik type operator for (\ref{HJ-Delta})).
\begin{Thm}\label{main1}
There exists $\lambda_1>0$ (depending on $T$, $r$ and $R$, but independent of $\Delta$) for which the following statements hold for any small $\Delta=(\dx,\dt)$ with $\lambda=\dt/\dx<\lambda_1$:
\begin{enumerate} 
\item[(1)] For each $n$ and $l$ with $0<l+1\le k(T)$ such that $(x_n,t_{l+1})\in\tG$, the expectation 
$$E^{l+1}_n(\xi):=E_{\mu_n^{l+1,0}(\cdot;\xi)}\Big[ \sum_{0<k\le l+1}L(\gamma^k,t_{k-1},\xi^k_{m(\gamma^k)})\dt+v_{m(\gamma^0)}^0 \Big]+ht_{l+1}$$
with respect to the probability measure of $\Omega_n^{l+1,0}$ has the infimum within all controls $\xi:G_n^{l+1,1}\to[-(d\lambda)^{-1},(d\lambda^{-1})]^d$. There exists the unique minimizing control $\xi^\ast$ for the infimum, which satisfies $| \xi^\ast{}^j|\le (d\lambda_1)^{-1}<(d\lambda)^{-1}$ for all $1\le j\le d$.
\item[(2)]Define the function $v$ on $\tG|_{0\le k\le k(T)}$ as $v(x_m,t_{k+1}):=\inf_\xi E^{k+1}_m(\xi)$ and $v(x_{m+\1},t_0):=v_{m+\1}^0$. Then, the minimizing control $\xi^\ast$ for $\inf_\xi E^{l+1}_n(\xi)$ satisfies
$$\xi^\ast{}^{k+1}_m=H_p(x_m,t_k,(D_xv)^k_m)\,\,\,\,(\Leftrightarrow\,\,(D_xv)^k_m=L_\xi(x_m,t_k,\xi^\ast{}^{k+1}_m)).$$
In particular, $(D_xv)^k_m$ is uniformly bounded on $\tG|_{0\le k\le k(T)}$ independently from $\Delta$ (a CFL-type condition).
\item[(3)] $v=v^k_{m+\1}$ defined in (2) is the solution of (\ref{HJ-Delta}).
\item[(4)] For each  $\omega\in B$,  let $\xi^\ast(\omega):G_{n+\1+\omega}^{l+1,1}\to[-(d\lambda)^{-1},(d\lambda^{-1})]^d$ be the minimizing control for $\inf_\xi E^{l+1}_{n+\1+\omega}(\xi)$. Let $\gamma\in \Omega_{n+\1+\omega}^{l+1,0}$ be the minimizing random walk generated by $\xi^\ast(\omega)$. Then, we have for $j=1,\ldots,d$,
\begin{eqnarray*}  
(D_{x^j}v)^{l+1}_{n+\1}&\le& E_{\mu_{n+\1-e_j}^{l+1,0}(\cdot;\xi^\ast(-e_j))}\Big[  \sum_{0<k\le l+1} L_{x^i}(\gamma^k,t_{k-1},\xi^k_{m(\gamma^{k})}(-e_j))\dt\\
&&+ (D_{x^j}v)^{0}_{m(\gamma^0)+e_j}\Big]+\theta \dx,\\
(D_{x^j}v)^{l+1}_{n+\1}&\ge&E_{\mu_{n+\1+e_j}^{l+1,0}(\cdot;\xi^\ast(e_j))}\Big[  \sum_{0<k\le l+1} L_{x^i}(\gamma^k,t_{k-1},\xi^k_{m(\gamma^{k})}(e_j)) \dt\\
&&+ (D_{x^j}v)^{0}_{m(\gamma^0)-e_j}\Big]-\theta \dx,
\end{eqnarray*}
where $\theta>0$ is independent of $\Delta$.
\end{enumerate}
\end{Thm}
We define the interpolation $v_\Delta(x,t):\R^d\times[0,t_{k(T)+1})\to\R$ of the difference solution $v^k_{m+\1}$ to  (\ref{HJ-Delta}), $u^j_\Delta(x,t):\R^d\times[0,t_{k(T)+1})\to\R^d$  of its difference partial derivative $(D_{x^j}v)^k_m$ and  $\gamma_\Delta(s):[0,t]\to\R^d$ of each sample path $\gamma\in \Omega_n^{l+1,0}$ with $t\in[t_{l+1},t_{l+2})$:
\begin{eqnarray*}
&&v_\Delta(x,t):=v^k_{m+\1}\\
&&\qquad \mbox{\quad for $x\in[x^1_{m^1+\1^1-1},x^1_{m^1+\1^1+1})\times\cdots\times[x^d_{m^d+\1^d-1},x^d_{m^d+1^d+1})$, $t\in[t_k,t_{k+1})$},\\
&&u_\Delta^j(x,t):=(D_{x^j}v)_{m}^k\\
&&\qquad  \mbox{\quad for $x\in[x^1_{m^1-1},x^1_{m^1+1})\times\cdots\times[x^d_{m^d-1},x^d_{m^d+1})$, $t\in[t_k,t_{k+1})$},\\
&&\gamma_\Delta(s):=
\left\{
\begin{array}{lll}
&x_n \mbox{\quad for  $s\in[t_{l+1},t]$},
\medskip\\
&\gamma^k+\frac{\gamma^{k+1}-\gamma^k}{\dt}(s-t_k) \mbox{\quad for  $s\in[t_{k},t_{k+1}]$}.
\end{array}
\right.
\end{eqnarray*}
The next theorem shows  convergence of approximation. 
\begin{Thm}\label{main2}
Take the limit $\Delta\to0$ under hyperbolic scaling, namely, $\Delta\to0$ with $0<\lambda_0\le\lambda=\dt/\dx<\lambda_1$, where $\lambda_1$ is the one mentioned in  Theorem \ref{main1} and $\lambda_0$ is a constant. Then, the following statements hold:
\begin{enumerate}
\item[(1)] Let $v$ be the viscosity solution of (\ref{HJ}) and let Let $v_\Delta$ be the interpolation of the  solution of (\ref{HJ-Delta}). 
Then, there exists $\beta>0$ independent of $\Delta$ and $v^0$ for which we have  
$$\sup_{\R^d\times[0,T]}|v_\Delta-v|\le \beta \sqrt{\dx} \mbox{ \,\,\,\,\,\,as $\Delta\to0$}.$$
\item[(2)] Let $(x,t)\in\R^d\times(0,T]$ be  regular and $\gamma^\ast:[0,t]\to\R^d$ be the minimizing curve for $v(x,t)$. Let $(x_n,t_{l+1})\in\tG|_{0\le k\le k(T)}$ be a point such that $t\in[t_{l+1},t_{l+2})$ and $x\in[x^1_{n^1-1},x^1_{n^1+1})\times\cdots\times[x^d_{n^d-1},x^d_{n^d+1})$. Let $\gamma_\Delta$ be the  interpolation  of the random walk $\gamma\in \Omega_n^{l+1,0}$ generated by the minimizing control for  $v^{l+1}_n$. Then, we have 
$$\gamma_\Delta\to\gamma^\ast\mbox{\quad uniformly on $[0,t]$ in probability \mbox{ \,\,\,\,\,\,as $\Delta\to0$}.}$$
In particular, the average of $\gamma_\Delta$ converges uniformly to $\gamma^\ast$ on $[0,t]$.
\item[(3)] 
Suppose that $v^0$ is semiconcave with a linear modulus. 
Let $u^j_\Delta$ be the interpolation  of the difference partial derivatives of the solution to (\ref{HJ-Delta}).  Then, for each regular point $(x,t)\in\R^d\times(0,T]$, we have  
 $$u_\Delta^j(x,t)\to v_{x^j}(x,t) \mbox{ \,\,\,\,\,\,as $\Delta\to0$ \quad($j=1,\ldots,d$)}.$$
In particular, $u^j_\Delta$ converges to $v_{x^j}$ pointwise a.e., and hence, for each compact set $K\subset\R^d$, $u^j_\Delta(\cdot,t)$ converges to $v_{x^j}(\cdot,t)$ in $L^1(K)$; $u^j_\Delta$ converges uniformly to $v_{x^j}$ on $(K\times[0,T])\setminus\Theta$, where $\Theta$ is a neighborhood of the set of points of singularity of $v_x$  with arbitrarily small measure.   
\item[(4)] 
Take the form \eqref{initial-2} of initial data, where $v^0$ is NOT supposed to be  semiconcave. 
Let $u^j_\Delta$ be the interpolation  of the difference partial derivatives of the solution to (\ref{HJ-Delta}).  Then, for each compact set $K\subset \R^d$ and each $t\in[0,T]$, we have  
 $$u_\Delta^j(\cdot,t)\to v_{x^j}(\cdot,t) \mbox{ in $L^1(K)$ \,\,\,\,\,\,as $\Delta\to0$ \quad($j=1,\ldots,d$)}.$$
\end{enumerate}
\end{Thm}
\begin{Rem}\label{remark22}
If $d=1$, we do not need the assumption of semiconcavity in (3).    
\end{Rem}
\noindent Although Remark \ref{remark22} is already announced in the work \cite{Soga2} of the author, we refer to more details later.   
\setcounter{section}{2}
\setcounter{equation}{0}
\section{Proof of result}
First, we state some properties of $L(x,t,\xi)$. 
\begin{Lemma}
Let $H(x,t,p)$ satisfy (H1)--(H4) and let $L(x,t,\xi)$ be the Legendre transform of $H(x,t,\cdot)$. Then, $L$ satisfies the following (L1)-(L4): 
\begin{enumerate}
\item[(L1)] $L(x,t,\xi):\R^d\times\R\times\R^d\to\R$, $C^2$,
\item[(L2)] $L_{\xi\xi}(x,t,\xi)$ is positive definite on $\R^d\times\R\times\R^d$,
\item[(L3)] $L$ is uniformly superlinear with respect to $\xi$, namely, for each $a\ge0$ there exists $b_2(a)\in\R$ such that $L(x,t,\xi)\ge a\norm \xi\norm+b_2(a)$  on $\R^d\times\R\times\R^d$,
\item[(L4)] $L, L_{x^i},L_{\xi^i}, L_{x^ix^j},L_{x^i\xi^j},L_{\xi^i\xi^j}$ are uniformly bounded on $\R^d\times\R\times K$ for each compact set $K\subset \R^d$ for $i,j=1,\ldots,d$.
\end{enumerate} 
\end{Lemma}
\begin{proof}
Minor variation of the reasoning in Chapter 1 of \cite{Fathi-book} yields the assertion. For readers' convenience, we give a proof. Fix $(x,t,\xi)\in\R^d\times\R\times\R^d$  arbitrarily. Take $a>\norm\xi\norm$. By (H3), we have   
$\xi\cdot p-H(x,t,p)\le \xi\cdot p-(a\norm p\norm+b_1(a))
\le(\norm\xi\norm-a)\norm p\norm-b_1(a)
\to-\infty$ as $\norm p\norm\to+\infty$. 
Hence, $\sup_{p\in\R^d}\{\xi\cdot p-H(x,t,p)\}$ is achieved in a bounded ball of $\R^d$, namely, for each $\xi$,  there exists $p^\ast\in\R^d$ such that $L(x,t,\xi)=\xi\cdot p^\ast-H(x,t,p^\ast)$ and $H_p(x,t,p^\ast)=\xi$. Since $H_p(x,t,p^\ast)=\xi$ is invertible with respect to $p^\ast$ due to (H2), we have the $C^1$-map $p^\ast=p(x,t,\xi)$ such that $H_p(x,t,p(x,t,\xi))\equiv\xi$.  Therefore, $L$ is of $C^1$. Furthermore, direct computation yields 
\begin{eqnarray*}
L_\xi(x,t,\xi)&=&p(x,t,\xi) \in C^1,\\
L_{\xi\xi}(x,t,\xi)&=&p_\xi(x,t,\xi)=H_{pp}(x,t,p(x,t,\xi))^{-1},\\
L_x(x,t,\xi)&=&-H_x(x,t,p(x,t,\xi))\in C^1,\\
L_{xx}(x,t,\xi)&=&-H_{xx}(x,t,p(x,t,\xi))\\
&&+H_{xp}(x,t,p(x,t,\xi))H_{pp}(x,t,p(x,t,\xi))^{-1}H_{px}(x,t,p(x,t,\xi)),\\
L_{x\xi}(x,t,\xi)&=&-H_{xp}(x,t,p(x,t,\xi))H_{pp}(x,t,p(x,t,\xi))^{-1}.  
\end{eqnarray*}
(L1) and (L2) are proved. 

For each $a\ge0$, define $b_2(a):=-\max_{x\in\R^d,t\in\R,\norm p\norm=a}H(x,t,p)$. We have $L(x,t,\xi)\ge \xi\cdot p-H(x,t,p)$ for all $x,t,\xi,p$. For each $\xi\in\R^d$, take $p:=a \xi/\norm \xi \norm$. Then, we see that $L(x,t,\xi)\ge a\norm \xi \norm +b_2(a)$, which holds for all $x,t,\xi$. (L3) is proved.

Let $K$ be a compact subset of $\R^d$ and let $a>\max_{\xi\in K}\norm \xi\norm$. Then, we have $H(x,t,p)\ge a\norm p\norm+b_1(a)$ for all $x,t,p$. Hence, $\xi\cdot p-H(x,t,p)\le \norm\xi \norm \norm p\norm -( a \norm p\norm+b_1(a) )<-b_1(a)$ for all $\xi\in K$ and $(x,t,p)\in\R^d\times\R\times\R^d$. Therefore, with (L3), we obtain $b_2(a)\le L(x,t,\xi)=\xi\cdot p(x,t,\xi)-H(x,t,p(x,t,\xi))< -b_1(a)$ on $\R^d\times\R\times K$.  Finally, we show the boundedness of $p(x,t,\xi)$ on $\R^d\times\R\times K$, which ends the proof of (L4) with the above expressions of the derivatives of $L$.  Suppose that there exist $(x_j,t_j,\xi_j)\in \R^d\times\R\times K$ for which the norm of $p_j:=p(x_j,t_j,\xi_j)$ goes to $+\infty$ as $j\to \infty$. Then, it holds that  $L(x_j,t_j,\xi_j)= \xi_j\cdot p_j-H(x_j,t_j,p_j)\le \norm  \xi_j\norm \norm p_j\norm -(a\norm p_j\norm+b_1(a))=(\norm  \xi_j\norm -a)\norm p_j\norm-b_1(a)\to-\infty$ as $j\to\infty$.
However, $L$ is bounded below by $b_2(a)$ on $\R^d\times\R\times K$, and we reach  a contradiction.
\end{proof}

\begin{proof}[{\it Proof of Theorem \ref{main1}.}]
We recall and define the following constants: 
\begin{eqnarray*}
&&\sup_{x\in\R^d}|v^0(x)|\le R,\qquad {\rm ess\,}\sup_{x\in\R^d}\norm v^0_{x}(x)\norm_\infty \le r,\\ &&\norm L_{x}(x,t,\xi)\norm_\infty\le \alpha (1+|L(x,t,\xi)|),\qquad L_\ast:=\min \{0,\,\,\inf_{x\in\R^d,t\in\R,\xi\in\R^d} L(x,t,\xi)\},\\
&&\alpha_1:=T\max_{x\in\R^d,t\in\R}|L(x,t,0)|+R,\quad
 \alpha_2:=\alpha\{\alpha_1+R+(1+2L_\ast)T\},\\
&&\lambda_1:=(d\max_{x\in\R^d,t\in\R,\norm p\norm_\infty\le 1+r+\alpha_2}\norm H_{p}(x,t,p)\norm_\infty)^{-1},\\
&& \theta:=T\max_{x\in\R^d,t\in\R,\norm \xi\norm_\infty\le (d\lambda_1)^{-1},i,j}|L_{x^ix^j}(x,t,\xi)|, 
\end{eqnarray*}
where $\norm x\norm_\infty:=\max_{1\le j\le d} |x^j|$ for $x\in\R^d$. Let $\Delta=(\dx,\dt)$ be such that $\dx \theta\le 1$ and $\lambda=\dt/\dx<\lambda_1$. 

Since $L$ and $v^0$ are bounded below, there exists the infimum of $E^{l+1}_n(\xi)$ with respect to $\xi:G_n^{l+1,1}\to[-(d\lambda)^{-1},(d\lambda)^{-1}]^d$ for each $n$ and $0\le l\le k(T)$. Since $G_n^{l+1,1}$ consists of a finite number of points, the compactness of a finite dimensional Euclidean space  implies that there exists a minimizing control  $\xi^\ast$ that attains the infimum. 

We have the equality for each $0<k\le l$,
\begin{eqnarray}\label{dynmprog}
E^{l+1}_{n}(\xi)
&=&\sum_{\gamma\in\Omega^{l+1,k}_n}\mu_n^{l+1,k}(\gamma;\xi|_{G_n^{l+1,k+1}})\Big(\sum_{k<k'\le l+1}L(\gamma^{k'},t_{k'-1},\xi^{k'}_{m(\gamma^{k'})})\dt\\\nonumber
&&+E^k_{m(\gamma^k)}(\xi|_{G_{m(\gamma^k)}^{k,1}})\Big)+h(t_{l+1}-t_k).
\end{eqnarray}
In order to check (\ref{dynmprog}), set $\tilde{\gamma}:=\gamma|_{k\le k'\le l+1}\in G_n^{l+1,k}$ and $\hat{\gamma}:=\gamma|_{0\le k'\le k}\in G_{m(\tilde{\gamma}^k)}^{k,0}$ for each $\gamma\in G_n^{l+1,0}$. Then, we have $\mu^{l+1,0}_n(\gamma)=\mu^{l+1,k}_n(\tilde{\gamma})\mu^{k,0}_{m(\tilde{\gamma}^k)}(\hat{\gamma})$. 
Hence, it follows from the definition of the random walk that  we have for each $0<k\le l$,
\begin{eqnarray*}
E^{l+1}_{n}(\xi)&=& \sum_{\gamma\in \Omega_n^{l+1,0}}\mu_n^{l+1,0}(\gamma;\xi)\Big( \sum_{0<k\le l+1}L(\gamma^k,t_{k-1},\xi^k_{m(\gamma^k)})\dt+v^0_{m(\gamma^0)} \Big)+ht_{l+1}\\\nonumber
&=&  \sum_{\tilde{\gamma}\in G_n^{l+1,k}} \mu_n^{l+1,k}(\tilde{\gamma}; \xi|_{G_n^{l+1,k+1}}) \Big\{    \sum_{\hat{\gamma}\in  {G_{m(\tilde{\gamma}^k)}^{k,0} } } 
 \mu^{k,0}_{m(\tilde{\gamma}^k)}(\hat{\gamma};   \xi|_{G_{m(\tilde{\gamma}^k)}^{k,1}})
  \\ 
&&\times \Big(\sum_{k<k'\le l+1}L(\tilde{\gamma}^{k'},t_{k'-1},\xi^{k'}_{m(\tilde{\gamma}^{k'})})\dt 
+\sum_{0<k'\le k}L(\hat{\gamma}^{k'},t_{k'-1},\xi^{k'}_{m(\hat{\gamma}^{k'})})\dt
 \\  \nonumber
&&+v^0_{m(\hat{\gamma}^0)}\Big ) \Big\} +ht_k+h(t_{l+1}-t_k), \\
&=&  \sum_{\tilde{\gamma}\in G_n^{l+1,k}} \mu_n^{l+1,k}(\tilde{\gamma}; \xi|_{G_n^{l+1,k+1}}) \Big[  \sum_{k<k'\le l+1}L(\tilde{\gamma}^{k'},t_{k'-1},\xi^{k'}_{m(\tilde{\gamma}^{k'})})\dt  \\
&&+\Big\{ \sum_{\hat{\gamma}\in  {G_{m(\tilde{\gamma}^k)}^{k,0} } } 
 \mu^{k,0}_{m(\tilde{\gamma}^k)}(\hat{\gamma};   \xi|_{G_{m(\tilde{\gamma}^k)}^{k,1}})
   \Big(\sum_{0<k'\le k}L(\hat{\gamma}^{k'},t_{k'-1},\xi^{k'}_{m(\hat{\gamma}^{k'})})\dt
\\
&&+ v^0_{m(\hat{\gamma}^0)}\Big)+ht_k \Big\}\Big]  +h(t_{l+1}-t_k),
\end{eqnarray*}
which implies (\ref{dynmprog}).

We observe that for each $n$ and any $\xi$, 
\begin{eqnarray}\label{looooong}
\,\,\,\,E^1_n(\xi)&=&L(x_n,t_0,\xi^1_n)\dt+\sum_{\omega\in B}\rho^1_n(\omega)v^0_{n+\omega}+ht_1\\\nonumber
&=&L(x_n,t_0,\xi^1_n)\dt+\frac{1}{2d}\sum_{\omega\in B} v^0_{n+\omega}
-\frac{\lambda}{2}\sum_{\omega\in B}(\omega\cdot \xi_n^1)v^0_{n+\omega}+h\dt\\\nonumber
&=&L(x_n,t_0,\xi^1_n)\dt+\frac{1}{2d}\sum_{\omega\in B} v^0_{n+\omega}\\\nonumber
&&-\frac{\dt}{2\dx}\sum_{i=1}^d \xi_n^1{}^i(v^0_{n+e_i}    -v^0_{n-e_i})+h\dt\\\nonumber
&=&L(x_n,t_0,\xi^1_n)\dt+\frac{1}{2d}\sum_{\omega\in B} v^0_{n+\omega}-\dt\xi_n^1\cdot(D_xv)^0_n+h\dt\\\nonumber
&=&-( \xi^1_n\cdot (D_{x}v)^0_n -L(x_n,t_0,\xi^1_n))\dt+\frac{1}{2d}\sum_{\omega\in B} v^0_{n+\omega}+h\dt\\\nonumber
&\ge& -H(x_n,t_0,(D_{x}v)^0_n)\dt +\frac{1}{2d}\sum_{\omega\in B} v^0_{n+\omega}+h\dt,
\end{eqnarray}
where the last inequality becomes an equality if and only if $\xi$ is given as 
$$\xi^1_n=H_p(x_n,t_0,(D_{x}v)^0_n).$$
Hence, the minimizing control $\xi^\ast$ for $v^1_n=\inf_\xi E^1_n(\xi)$ satisfies 
\begin{eqnarray}\label{mini1}
\xi^\ast{}^1_n=H_p(x_n,t_0,(D_{x}v)^0_n),\quad \norm \xi^\ast\norm_\infty\le (d\lambda_1)^{-1}
\end{eqnarray}
with  
\begin{eqnarray}\label{eqn1}
 v^1_n&=&-H(x_n,t_0,(D_{x}v)^0_n)\dt +\frac{1}{2d}\sum_{\omega\in B} v^0_{n+\omega}+h\dt\\\nonumber
  \qquad &\Leftrightarrow&  (D_t v)_n^1+H(x_n,t_0,(D_xv)^0_n)=h.
 \end{eqnarray}
\indent We proceed by induction. Suppose the following statement:
\begin{itemize}
\item[ (A)$_l$] For some $l\ge0$, the minimizing control $\xi^\ast$ for $v^{l+1}_n=\inf_\xi E_n^{l+1}(\xi)$ is uniquely obtained for each $n$ as $\xi^\ast{}^{k+1}_{m}=H_p(x_m,t_k,(D_{x}v)^k_m)$ on $G_n^{l+1,1}$ with $\norm \xi^\ast{}^k_{m+\1}\norm_\infty\le (d\lambda_1)^{-1}.$
\end{itemize}
(A)$_{l=0}$ is true.   In order to see that (A)$_{l+1}$ is also true, we first examine the bound of $(D_xv)^{l+1}_{n+\1}$ for each $n$.  
Let $\xi^\ast(\omega)$ denote the unique minimizer for $v^{l+1}_{n+\1+\omega}$, $\omega\in B$. The variational property yields for each $e_j$,
\begin{eqnarray*}
v^{l+1}_{n+\1-e_j}&=& E_{\mu^{l+1,0}_{n+\1-e_j}(\cdot;\xi^\ast(-e_j))}\Big[\sum_{0<k\le l+1}L(\gamma^{k},t_{k-1}, \xi^\ast{}^{k}_{m(\gamma^{k})}(-e_j))\dt +v^0_{m(\gamma^0)}  \Big]+ht_k,\\
v^{l+1}_{n+\1+e_j}&\le& E_{\mu^{l+1,0}_{n+\1-e_j}(\cdot;\xi^\ast(-e_j))}\Big[\sum_{0<k\le l+1}L(\gamma^{k}+e_j\cdot2\dx,t_{k-1}, \xi^\ast{}^{k}_{m(\gamma^{k})}(-e_j))\dt \\
&&+v^0_{m(\gamma^0+e_j\cdot2\dx)}  \Big] +ht_k.
\end{eqnarray*}
Hence, from $v^{l+1}_{n+\1+e_j}-v^{l+1}_{n+\1-e_j}$, we obtain
\begin{eqnarray*}
&&(D_{x^j}v)^{l+1}_{n+\1}\le E_{\mu^{l+1,0}_{n+\1-e_j}(\cdot;\xi^\ast(-e_j))}\Big[\sum_{0<k\le l+1}L_{x^j}(\gamma^{k},t_{k-1}, \xi^\ast{}^{k}_{m(\gamma^{k})}(-e_j))\dt 
+(D_{x^j}v)^0_{m(\gamma^0)+e_j}\Big]\\
&&\qquad +\theta \dx\\
&&\le  E_{\mu^{l+1,0}_{n+\1-e_j}(\cdot;\xi^\ast(-e_j))}\Big[\sum_{0<k\le l+1}L_{x^j}(\gamma^{k},t_{k-1}, \xi^\ast{}^{k}_{m(\gamma^{k})}(-e_j))\dt \Big]+r+\theta \dx\\
&&\le  E_{\mu^{l+1,0}_{n+\1-e_j}(\cdot;\xi^\ast(-e_j))}\Big[\sum_{0<k\le l+1}\alpha \{1+|L(\gamma^{k},t_{k-1}, \xi^\ast{}^{k}_{m(\gamma^{k})}(-e_j))|\}\dt \Big]+r+\theta \dx\\
&&\le  E_{\mu^{l+1,0}_{n+\1-e_j}(\cdot;\xi^\ast(-e_j))}\Big[\sum_{0<k\le l+1}\alpha \{1+L(\gamma^{k},t_{k-1}, \xi^\ast{}^{k}_{m(\gamma^{k})}(-e_j))+2|L_\ast|\}\dt \Big]+r+\theta \dx\\
&&=  E_{\mu^{l+1,0}_{n+\1-e_j}(\cdot;\xi^\ast(-e_j))}\Big[\sum_{0<k\le l+1}\alpha \{1+L(\gamma^{k},t_{k-1}, \xi^\ast{}^{k}_{m(\gamma^{k})}(-e_j))+2|L_\ast|\}\dt\\
&&\quad +\alpha\{v^0_{m(\gamma^0)} -v^0_{m(\gamma^0)}\} \Big] +r+\theta \dx\\
&&\le \alpha E_{\mu^{l+1,0}_{n+\1-e_j}(\cdot;\xi^\ast(-e_j))}\Big[\sum_{0<k\le l+1} L(\gamma^{k},t_{k-1}, \xi^\ast{}^{k}_{m(\gamma^{k})}(-e_j))\dt
+v^0_{m(\gamma^0)} \Big]\\
&&\quad+\alpha T+2\alpha|L_\ast|T+\alpha R+r+\theta \dx\\
&&\le \alpha E_{\mu^{l+1,0}_{n+\1-e_j}(\cdot;\xi^\ast(-e_j))}\Big[\sum_{0<k\le l+1} L(\gamma^{k},t_{k-1}, 0)\dt+v^0_{m(\gamma^0)} \Big]\\
&&\quad+\alpha T+2\alpha|L_\ast|T+\alpha R+r+\theta \dx\\
&&\le\alpha\alpha_1+\alpha T+2\alpha|L_\ast|T+\alpha R+r+1 \\
&&=\alpha_2+r+1.
\end{eqnarray*}
We can show $(D_{x^j}v)^{l+1}_{n+\1}\ge -\alpha_2-r-1$ in a similar way with the minimizing control for $v^{l+1}_{n+\1+e_j}$. 
Therefore, we obtain 
\begin{eqnarray}\nonumber
&&\norm(D_{x}v)^{l+1}_{n+\1}\norm_\infty\le 1+r+\alpha_2,\\\label{bounded}
&&\norm H_p(x_{n+\1},t_{l+1},(D_{x}v)^{l+1}_{n+\1})\norm_\infty\le (d\lambda_1)^{-1}.
\end{eqnarray} 
With (\ref{dynmprog}) and similar calculation in (\ref{looooong}), we observe that for each $n$ and for any $\xi$, 
\begin{eqnarray}\nonumber
\quad E^{l+2}_{n+\1}(\xi)&=& L(x_{n+\1},t_{l+1},\xi^{l+2}_{n+\1})\dt +\sum_{\omega\in B}\rho^{l+2}_{n+\1}(\omega) E^{l+1}_{n+\1+\omega}(\xi|_{G_{n+\1+\omega}^{l+1,1}})+h\dt\\\label{unique}
&\ge&L(x_{n+\1},t_{l+1},\xi^{l+2}_{n+\1})\dt +\sum_{\omega\in B}\rho^{l+2}_{n+\1}(\omega) v^{l+1}_{n+\1+\omega}+h\dt\\\nonumber
&=&L(x_{n+\1},t_{l+1},\xi^{l+2}_{n+\1})\dt +\frac{1}{2d}\sum_{\omega\in B}v^{l+1}_{n+\1+\omega}-\xi^{l+2}_{n+\1}\cdot(D_xv)^{l+1}_{n+\1}\dt+h\dt\\\label{unique2}
&\ge&-H(x_{n+\1},t_{l+1},(D_{x}v)^{l+1}_{n+\1})\dt +\frac{1}{2d}\sum_{\omega\in B}v^{l+1}_{n+\1+\omega}+h\dt.
\end{eqnarray}
It follows from the assumption (A)$_l$ of induction and the properties of the Legendre transform that the above inequalities (\ref{unique}) and (\ref{unique2}) become equalities, if $\xi=\xi^\ast$ is given as 
\begin{eqnarray}\label{mini2}
\xi^\ast{}^{k+1}_{m}=H_p(x_{m},t_{k},(D_{x}v)^{k}_{m})\mbox{ \quad on $G_{n+\1}^{l+2,1}$},
\end{eqnarray}
which makes sense due to (\ref{bounded}). Hence, we obtain a minimizing control $\xi^\ast$ for $v^{l+2}_{n+\1}=\inf_\xi E^{l+2}_{n+\1}(\xi)$ with $\norm \xi^\ast\norm_\infty \le (d\lambda_1)^{-1}$ and 
\begin{eqnarray}\label{eqn2}
v^{l+2}_{n+\1} &=&-H(x_{n+\1},t_{l+1},(D_{x}v)^{l+1}_{n+\1})\dt +\frac{1}{2d}\sum_{\omega\in B}v^{l+1}_{n+\1+\omega}+h\dt\\\nonumber
&\Leftrightarrow&  (D_tv)^{l+2}_{n+\1}+H(x_{n+\1},t_{l+1},(D_{x}v)^{l+1}_{n+\1})=h.
\end{eqnarray}
\indent Let $\xi$ be another minimizing control for $v^{l+2}_{n+\1}$ different from the above $\xi^\ast$. If $\xi^{l+2}_{n+\1}\neq\xi^\ast{}^{l+2}_{n+\1}$, then (\ref{unique2}) becomes a strict inequality, yielding the contradiction that $v^{l+2}_{n+\1}>v^{l+2}_{n+\1}$. If $\xi^{l+2}_{n+\1}=\xi^\ast{}^{l+2}_{n+\1}$, then there necessarily exists $\omega\in B$ such that $\xi|_{G_{n+\1+\omega}^{l+1,1}}$ is not the minimizing control for $v_{n+\1+\omega}^{l+1}$, because of the assumption (A)$_l$ of our induction; hence (\ref{unique}) becomes a strict inequality, yielding the same contradiction.   

Thus, we conclude that the minimizing control $\xi^\ast$ for $v^{l+2}_{n+\1}(\xi)=\inf_\xi E^{l+2}_{n+\1}(\xi)$ is uniquely obtained for each $n$ as (\ref{mini2}) with $\norm \xi^\ast\norm_\infty \le (d\lambda_1)^{-1}$. 
By induction, (1) is clear. (2) follows from (\ref{mini1}) and (\ref{mini2}). (3) follows from (\ref{eqn1}) and (\ref{eqn2}). The inequalities in (4) are obtained in the above calculation to derive  (\ref{bounded}). 
\end{proof}
 Next, we study the hyperbolic scaling limit of our random walks in order to prove Theorem \ref{main2}. Let $\bar{\gamma}^k$ be the averaged path of $\gamma\in \Omega_n^{l+1,0}$ generated by a control $\xi$, i.e.,
$$\bar{\gamma}^k:=\sum_{\gamma\in\Omega^{l+1,0}_n}\mu^{l+1,0}_n(\gamma)\gamma^k,\quad 0\le k\le l+1,$$ 
where we use the short notation $\mu^{l+1,l'}_n(\cdot)$ instead of $\mu^{l+1,l'}_n(\cdot;\xi)$.  We see that $\bar{\gamma}$ satisfies 
$$\bar{\gamma}^k=\bar{\gamma}^{k+1}-\bar{\xi}^{k+1}\dt\,\,\,
{\rm \quad with \quad}\bar{\xi}^k:=\sum_{\gamma\in\Omega^{l+1,0}_n}\mu^{l+1,0}_n(\gamma)\xi^k_{m(\gamma^k)}.$$ 
In fact, we have 
\begin{eqnarray*}
\bar{\gamma}^k&=&\sum_{\gamma\in\Omega^{l+1,0}_n}\mu^{l+1,0}_n(\gamma)\gamma^k
=\sum_{\gamma\in\Omega^{l+1,k}_n}\mu^{l+1,k}_n(\gamma)\gamma^k\\
&=&\sum_{\gamma\in\Omega^{l+1,k+1}_n}\sum_{\omega\in B}\mu^{l+1,k+1}_n(\gamma)\rho^{k+1}_{m(\gamma^{k+1})}(\omega)(\gamma^{k+1}+\omega\dx)\\
&=&\bar{\gamma}^{k+1}+\sum_{\gamma\in\Omega^{l+1,k+1}_n}\sum_{\omega\in B}\mu^{l+1,k+1}_n(\gamma)  \Big( \frac{1}{2d}-\frac{\lambda}{2}\omega\cdot\xi^{k+1}_{m(\gamma^{k+1})}\Big)\omega\dx\\
&=&\bar{\gamma}^{k+1}-\sum_{\gamma\in\Omega^{l+1,k+1}_n}\sum_{\omega\in B}\mu^{l+1,k+1}_n(\gamma) \big(\omega\cdot\xi^{k+1}_{m(\gamma^{k+1})}\big) \omega\frac{\lambda}{2}\dx\\
&=&\bar{\gamma}^{k+1}-\frac{1}{2}\sum_{\omega\in B}(\omega\cdot\bar{\xi}^{k+1})\omega\dt\\
&=&\bar{\gamma}^{k+1}-\frac{1}{2}\sum_{i=1}^d\{(e_i\cdot\bar{\xi}^{k+1})e_i+(-e_i\cdot\bar{\xi}^{k+1}) (-e_i)\}\dt\\
&=&\bar{\gamma}^{k+1}-\bar{\xi}^{k+1}\dt. 
\end{eqnarray*}
\indent Let $\eta(\gamma)$ be a random variable defined for each $\gamma\in\Omega^{l+1,0}_n$ as
$$\eta^k(\gamma)=\eta^{k+1}(\gamma)-\xi^{k+1}_{m(\gamma^{k+1})}\dt,\,\,\,\eta^{l+1}(\gamma)=x_n.$$
Let $\eta_\Delta(\gamma)(\cdot):[0,t]\to\R^d$ denote the linear interpolation of $\eta(\gamma)$ with $t\in[t_{l+1},t_{l+2})$,
\begin{eqnarray*}
\eta_\Delta(\gamma)(s):=
\left\{
\begin{array}{lll}
&x_n \mbox{\quad for  $s\in[t_{l+1},t]$},
\medskip\\
&\eta^k(\gamma)+\frac{\eta^{k+1}(\gamma)-\eta^k(\gamma)}{\dt}(s-t_k) \mbox{\quad for  $s\in[t_{k},t_{k+1}]$}.
\end{array}
\right.
\end{eqnarray*} 
Define  $\tilde{\sigma}^k_i$ and $\tilde{\delta}^k_i$ for $i=1,\ldots,d$ as 
$$\tilde{\sigma}^k_i:=E_{\mu^{l+1,0}_n(\cdot;\xi)}[|(\eta^k(\gamma)-\gamma^k)^i|^2],\quad \tilde{\delta}^k_i:=E_{\mu^{l+1,0}_n(\cdot;\xi)}[|(\eta^k(\gamma)-\gamma^k)^i|],$$
where $(\eta^k(\gamma)-\gamma^k)^i$ denotes the $i$-th component of $\eta^k(\gamma)-\gamma^k$. The following lemma is a key to the convergence of our difference scheme:
\begin{Lemma}\label{limit theorem} 
For any control $\xi$, we have 
$$(\tilde{\delta}^k_i)^2\le \tilde{\sigma}^k_i\le (t_{l+1}-t_k)\frac{\dx}{\lambda}\quad{\rm for}\quad  0\le k\le l+1.$$
\end{Lemma}
\begin{Rem}
$\tilde{\sigma}^k_i$ can be seen as a generalization of the standard variance. The standard variance is of $O(1)$ as $\Delta\to0$ under hyperbolic scaling in general for space-time  inhomogeneous random walks. However, $\tilde{\sigma}^k_i$ and $\tilde{\delta}^k_i$ always tend to $0$ for any control $\xi$ as $\Delta\to0$ under hyperbolic scaling. In the homogeneous case (i.e., $\xi$ is constant), $\tilde{\sigma}^k_i$ is equal to the standard variance. See also \cite{Soga1}.
\end{Rem}
\begin{proof}[proof of Lemma \ref{limit theorem} ]
We observe that 
\begin{eqnarray*}
\tilde{\sigma}^k_i&=&\sum_{\gamma\in\Omega^{l+1,k}_n}\mu^{l+1,k}_n(\gamma)|(\eta^k(\gamma)-\gamma^k)^i|^2\\
&=&\sum_{\gamma\in\Omega^{l+1,k+1}_n}\mu^{l+1,k+1}_n(\gamma) \sum_{\omega\in B}\rho^{k+1}_{m(\gamma^{k+1})}(\omega) \\
&&\times\big\{  (\eta^{k+1}(\gamma)-\gamma^{k+1})^i - (\xi^{k+1}_{m(\gamma^{k+1})})^i\dt-\omega^i\dx  \big\}^2    \\
&=&\tilde{\sigma}^{k+1}_{i} 
+ \sum_{\gamma\in\Omega^{l+1,k+1}_n}\mu^{l+1,k+1}_n(\gamma)\Big[  \sum_{\omega\in B} \rho^{k+1}_{m(\gamma^{k+1})}(\omega) \big\{ \lambda(\xi^{k+1}_{m(\gamma^{k+1})})^i+ \omega^i \big\}^2  \Big]\dx^2\\
&&- \sum_{\gamma\in\Omega^{l+1,k+1}_n}\mu^{l+1,k+1}_n(\gamma)\Big\{ 2\lambda(\eta^{k+1}(\gamma)-\gamma^{k+1} )^i(\xi^{k+1}_{m(\gamma^{k+1})})^i\\
&& + 2(\eta^{k+1}(\gamma)-\gamma^{k+1} )^i\sum_{\omega\in B}\rho^{k+1}_{m(\gamma^{k+1})}(\omega)\omega^i    \Big\}\dx.
\end{eqnarray*}
Since 
\begin{eqnarray*}
\sum_{\omega\in B}\rho^{k+1}_{m(\gamma^{k+1})}(\omega)\omega^i
&=& \sum_{j=1}^d\Big\{\Big(\frac{1}{2d}-\frac{\lambda}{2}e_j\cdot\xi^{k+1}_{m(\gamma^{k+1})}\Big)e_j^i+\Big(\frac{1}{2d}-\frac{\lambda}{2}(-e_j)\cdot\xi^{k+1}_{m(\gamma^{k+1})}\Big)(-e_j^i)\Big\}\\
&=&-\lambda (\xi^{k+1}_{m(\gamma^{k+1})})^i,
\end{eqnarray*}
we obtain 
\begin{eqnarray*}
\tilde{\sigma}^k_{i}&=&\tilde{\sigma}^{k+1}_{i} 
+ \sum_{\gamma\in\Omega^{l+1,k+1}_n}\mu^{l+1,k+1}_n(\gamma)\Big[  \sum_{\omega\in B} \rho^{k+1}_{m(\gamma^{k+1})}(\omega) \big\{  \omega^i+\lambda(\xi^{k+1}_{m(\gamma^{k+1})})^i \big\}^2  \Big]\dx^2\\
&=&\tilde{\sigma}^{k+1}_{i} 
+ \sum_{\gamma\in\Omega^{l+1,k+1}_n}\mu^{l+1,k+1}_n(\gamma)\Big[  \sum_{\omega\in B} \rho^{k+1}_{m(\gamma^{k+1})}(\omega) \big[ (\omega^i)^2+  2\lambda\omega^i(\xi^{k+1}_{m(\gamma^{k+1})})^i \\
&&+\{\lambda(\xi^{k+1}_{m(\gamma^{k+1})})^i\}^2 \big] \Big]\dx^2\\
&\le&\tilde{\sigma}^{k+1}_{i} 
+ \sum_{\gamma\in\Omega^{l+1,k+1}_n}\mu^{l+1,k+1}_n(\gamma)\Big[  1+\{\lambda(\xi^{k+1}_{m(\gamma^{k+1})})^i\}^2  \\
&&+  2\lambda(\xi^{k+1}_{m(\gamma^{k+1})})^i  \sum_{\omega\in B} \rho^{k+1}_{m(\gamma^{k+1})}(\omega) \omega^i
 \Big\}\dx^2\\
 &=& \tilde{\sigma}^{k+1}_{i} 
+ \sum_{\gamma\in\Omega^{l+1,k+1}_n}\mu^{l+1,k+1}_n(\gamma)\Big[   1-\{\lambda(\xi^{k+1}_{m(\gamma^{k+1})})^i\}^2  \Big]\dx^2\\
&\le&\tilde{\sigma}^{k+1}_{i} +\frac{\dx}{\lambda}\dt.
\end{eqnarray*}
Since $\tilde{\sigma}^{l+1}_{i}=0$, the assertion is proved. 
\end{proof}
The rate $O(\sqrt{\dx})$ of convergence of our scheme comes from the asymptotics of $\tilde{\delta}^k_i$ as we will see below.

We observe the following facts on the viscosity solution $v$ of (\ref{HJ}): 
\begin{Lemma}\label{lemma-regularity}
Let $\gamma^\ast:[0,t]\to\R^d$ be a minimizing curve for $v(x,t)$.
 \begin{enumerate}
 \item [(1)]The following regularity properties hold:
 \begin{eqnarray*}
&&L^c_{\xi}(\gamma^\ast(\tau),\tau,\gamma^\ast{}'(\tau))\in\p^-_xv(\gamma^\ast(\tau),\tau)\mbox{ for $0\le \tau<t$}, \\
&&L^c_{\xi}(\gamma^\ast(\tau),\tau,\gamma^\ast{}'(\tau))\in\p^+_xv(\gamma^\ast(\tau),\tau)\mbox{ for $0< \tau\le t$,}
\end{eqnarray*}
 where $\p^-_xv$ (resp., $\p^+_xv$) denotes the subdifferential (resp., superdifferential). In particular $v_x(\gamma^\ast(\tau),\tau)$ exists for $0<\tau<t$ and is equal to $L^c_{\xi}(\gamma^\ast(\tau),\tau,\gamma^\ast{}'(\tau))$.
 \item [(2)]$|\gamma^\ast{}'(\tau)|\le(d\lambda_1)^{-1}$ for $0\le \tau\le t$, where $\lambda_1$ is the number given in the proof of Theorem \ref{main1}.
 \item[(3)] If $(x,t)$ is  regular, we have for any $0\le\tau< t$
 $$v_x(x,t)=\int^t_\tau L^c_x(\gamma^\ast(s),s,\gamma^\ast{}'(s))ds +L^c_{\xi}(\gamma^\ast(\tau),\tau,\gamma^\ast{}'(\tau)).$$ 
 If the initial data $v^0$ is semiconcave, the superdifferential of $v^0$ is not empty and $L^c_{\xi}(\gamma^\ast(0),0,\gamma^\ast{}'(0))=v^0_x(\gamma^\ast(0))$. 
 \end{enumerate}
 \end{Lemma} 
\noindent  This lemma is well-known (a proof is given  in the same manner as the proof of Lemma 3.2 in \cite{Soga2}). 
%
%
\begin{proof} [{\it Proof of Theorem \ref{main2}.}] Hereafter, $\beta_1,\beta_2,\ldots$ are constants independent of $\dx,\dt,x_m,t_k$ and $v^0$ in (\ref{HJ}). 

We prove (1). Since the solution $v$ of (\ref{HJ}) is Lipschitz ((1) and (2) of Lemma \ref{lemma-regularity} imply a  Lipschitz constant independent of $v^0$), it is enough to show $|v^{l+1}_n-v(x_n,t_{l+1})|\le \beta\sqrt{\dx}$. Let $\gamma^\ast$ be a minimizing curve for $v(x_n,t_{l+1})$. Consider the  control $\xi$ defined on $G_n^{l+1,1}$ as
$$\xi(x_m,t_{k+1}):=\gamma^\ast{}'(t_{k+1})$$
and the random walk $\gamma$ generated by $\xi$. Then, $\eta(\gamma)$ is independent of $\gamma\in\Omega_n^{l+1,0}$ and satisfies
\begin{eqnarray*}
&&\norm \eta^k(\gamma)-\gamma^\ast{}(t_k) \norm_\infty\le \beta_1\dx\mbox{\,\,\, for all  $0\le k\le l+1$},\\
&&\Big|\int_0^{t_{l+1}}L(\gamma^\ast(s),s,\gamma^\ast{}'(s))ds+v^0(\gamma^\ast(0))\\
&&\qquad\qquad -\Big(  \sum_{0<k\le l+1}L(\eta^k(\gamma),t_{k-1},\xi^k_{m(\gamma^k)})\dt +v^0(\eta^0(\gamma)) \Big)\Big|\le \beta_2\dx.
\end{eqnarray*}
It follows from Lemma \ref{limit theorem} that 
\begin{eqnarray*}
v^{l+1}_n&\le& E_{\mu(\cdot;\xi)}\Big[ \sum_{0<k\le l+1}L(\gamma^k,t_{k-1},\xi^k_{m(\gamma^k)})\dt+v^0_{m(\gamma^0)}  \Big]+ht_{l+1}\\
&\le& E_{\mu(\cdot;\xi)}\Big[ \sum_{0<k\le l+1}L(\gamma^k,t_{k-1},\xi^k_{m(\gamma^k)})\dt+v^0(\gamma^0)  \Big]+ht_{l+1}+r\dx \\
&\le & E_{\mu(\cdot;\xi)}\Big[ \sum_{0<k\le l+1}L(\eta^k(\gamma),t_{k-1},\xi^k_{m(\gamma^k)})\dt+v^0(\eta^0(\gamma))  \Big]+ht_{l+1} +\beta_3\sqrt{\dx}\\
&\le& \int_0^{t_{l+1}}L(\gamma^\ast(s),s,\gamma^\ast{}'(s))ds+v^0(\gamma^\ast(0)) +ht_{l+1}+\beta_4\sqrt{\dx}. 
\end{eqnarray*}
Hence, we obtain 
$$v^{l+1}_n-v(x_n,t_{l+1})\le \beta_4\sqrt{\dx}.$$
Let $\xi^\ast$ be the minimizing control for $v^{l+1}_n$. Consider the linear interpolation of $\eta(\gamma)$ within $[0,t_{l+1}]$. Then, we have 
\begin{eqnarray*}
&&\eta_\Delta(\gamma)'(s)=\xi^\ast{}^k_{m(\gamma^k)}\mbox{\,\,\,  for $s\in(t_{k-1},t_k)$},\\
&&v(x_n,t_{l+1})\le\int_0^{t_{l+1}}L(\eta_\Delta(s),s,\eta_\Delta(\gamma)'(s))dt+v^0(\eta_\Delta(\gamma)(0))+ht_{l+1}\mbox{\,\,\, for $\gamma\in\Omega^{l+1,0}_n$},\\
&&v(x_n,t_{l+1})\le E_{\mu^{l+1,0}_n(\cdot;\xi^\ast)}\Big[\int_0^{t_{l+1}}L(\eta_\Delta(s),s,\eta_\Delta(\gamma)'(s))dt+v^0(\eta_\Delta(\gamma)(0))\Big]+ht_{l+1}.
\end{eqnarray*}
It follows from Lemma \ref{limit theorem} that 
\begin{eqnarray*}
v^{l+1}_n &=&  E_{\mu^{l+1,0}_n(\cdot;\xi^\ast)}\Big[ \sum_{0<k\le l+1}L(\gamma^k,t_{k-1},\xi^\ast{}^k_{m(\gamma^k)})\dt+v^0_{m(\gamma^0)}  \Big]+ht_{l+1}\\
&\ge&  E_{\mu^{l+1,0}_n(\cdot;\xi^\ast)}\Big[ \sum_{0<k\le l+1}L(\gamma^k,t_{k-1},\xi^\ast{}^k_{m(\gamma^k)})\dt+v^0(\gamma^0)  \Big]+ht_{l+1}-r\dx\\
&\ge& E_{\mu^{l+1,0}_n(\cdot;\xi^\ast)}\Big[ \sum_{0<k\le l+1}L(\eta^k(\gamma),t_{k-1},\xi^\ast{}^k_{m(\gamma^k)})\dt+v^0(\eta^0(\gamma))  \Big]\\
&&+ht_{l+1}-\beta_5\sqrt{\dx}\\
&\ge& E_{\mu^{l+1,0}_n(\cdot;\xi^\ast)}\Big[ \int _0^{t_{l+1}} L(\eta_\Delta(\gamma)(s),s,\eta_\Delta(\gamma)'(s))ds+v^0(\eta_\Delta(\gamma)(0))  \Big]\\
&&+ht_{l+1}-\beta_6\sqrt{\dx}.
\end{eqnarray*}
Therefore, we obtain
\begin{eqnarray*}
v^{l+1}_n-v(x_n,t_{l+1})&\ge& -\beta_6\sqrt{\dx}.
\end{eqnarray*}
Thus, (1) is proved.

In order to prove (2), we prepare two lemmas.
\begin{Lemma}\label{cal}
Let $\gamma^\ast$ be the unique minimizer for $v(x,t)$. Define the set $\Gamma^\varepsilon$ with $\varepsilon>0$ and $b>0$ as the family of all Lipschitz-curves $\nu:[0,t]\to\R^d$ such that 
\begin{eqnarray*}
&&\norm \nu(t)-\gamma^\ast(t) \norm_\infty\le \varepsilon,\quad \norm \nu'(s) \norm_\infty\le b\mbox{ for a.e. $s\in[0,t]$},\\
&&\int^t_0L(\nu(s),s,\nu'(s))ds+v^0(\nu(0))\le \int^t_0L(\gamma^\ast(s),s,\gamma^\ast{}'(s))ds+v^0(\gamma^\ast(0))+\varepsilon.
\end{eqnarray*}
 Then, we have as $\varepsilon\to0$,
$$\sup_{\nu\in\Gamma^\varepsilon}\norm \nu-\gamma^\ast\norm_{C^0([0,t])}\to0,\,\,\,\sup_{\nu\in\Gamma^\varepsilon}\norm \nu'-\gamma^\ast{}'\norm_{L^2([0,t])}\to0.$$
\end{Lemma}
\noindent Here, $\norm \nu\norm_{C^0([0,t])}:=\sup_{s\in[0,t]}\norm \nu(s) \norm $ and $\norm \nu\norm_{L^2([0,t])}:=\{\int_0^t \norm \nu(s) \norm^2 ds\}^\frac{1}{2}$.  Lemma \ref{cal} is proved in a similar way to the proof of Lemma 3.4 in \cite{Soga2}.

\begin{Lemma}\label{convert}
Let $f:[0,t]\to\R$ be a Lipschitz function with a Lipschitz constant $\theta$ satisfying $f(t)=0$. Then, it holds that $\norm f\norm_{C^0([0,t])}\le\theta\norm f \norm_{L^2([0,t])}+\sqrt{\norm f\norm_{L^2([0,t])}}$.
\end{Lemma}
\noindent See Lemma 3.5 of \cite{Soga2} for a proof. 

We prove (2). For each fixed $\varepsilon>0$, define the set 
\begin{eqnarray*}
\Omega^\varepsilon_\Delta:=\{ \gamma\in\Omega^{l+1,0}_n\,\,|\,\,\norm \gamma_\Delta-\gamma^\ast\norm_{C^0([0,t])}\le \varepsilon \}.
\end{eqnarray*}
Our task is  to prove that $prob(\Omega^\varepsilon_\Delta)\to1$ as $\Delta\to0$. We first obtain an estimate of $\norm\gamma_\Delta-\gamma^\ast\norm_{L^2([0,t])}$ and then convert it into that of $\norm \gamma_\Delta-\gamma^\ast\norm_{C^0([0,t])}$ using  Lemma \ref{convert}.
 
Observe that 
\begin{eqnarray}\label{star}
&&(E_{\mu^{l+1,0}_n(\cdot;\xi^\ast)}[ \norm\gamma_\Delta-\gamma^\ast\norm_{L^2([0,t])} ])^2\le 
E_{\mu^{l+1,0}_n(\cdot;\xi^\ast)}[ \norm\gamma_\Delta-\gamma^\ast\norm_{L^2([0,t])}^2 ]\\\nonumber
&&\qquad \le 2E_{\mu^{l+1,0}_n(\cdot;\xi^\ast)}[ \norm\gamma_\Delta-\eta_\Delta(\gamma)\norm_{L^2([0,t])}^2 ]+2E_{\mu^{l+1,0}_n(\cdot;\xi^\ast)}[ \norm\eta_\Delta(\gamma)-\gamma^\ast\norm_{L^2([0,t])}^2 ],
\end{eqnarray}
where $E_{\mu(\cdot;\xi^\ast)}[ \norm\gamma_\Delta-\eta_\Delta(\gamma)\norm_{L^2([0,t])}^2 ]$ tends to $0$ as $\Delta\to0$ due to Lemma \ref{limit theorem}. We show that $E_{\mu^{l+1,0}_n(\cdot;\xi^\ast)}[ \norm\eta_\Delta(\gamma)-\gamma^\ast\norm_{L^2([0,t])}^2 ]$ also tends to $0$ as $\Delta\to0$. For this purpose, define the  set 
$$\tilde{\Omega}^\varepsilon_\Delta:=\{ \gamma\in\Omega^{l+1,0}_n\,\,|\,\,\norm \eta_\Delta(\gamma)-\gamma^\ast\norm_{C^0([0,t])}\le \varepsilon,\,\,\,\norm \eta_\Delta(\gamma)'-\gamma^\ast{}'\norm_{L^2([0,t])}\le \varepsilon \}$$
for each fixed $\varepsilon>0$, and show $prob(\tilde{\Omega}^\varepsilon_\Delta)\to 1$ as $\Delta\to0$. It follows from Lemma \ref{limit theorem} that 
\begin{eqnarray*}
v^{l+1}_n&=&E_{\mu^{l+1,0}_n(\cdot;\xi^\ast)}\Big[ \int^t_0L(\eta_\Delta(\gamma)(s),s,\eta_{\Delta}(\gamma)'(s)  )ds+v^0(\eta_\Delta(\gamma)(0) ) \Big]\\
&&+ht_{l+1} +O(\sqrt{\dx}). 
\end{eqnarray*}
By (1), we have 
\begin{eqnarray*}
v^{l+1}_n-v(x,t)&=&O(\sqrt{\dx})\\
&=& E_{\mu^{l+1,0}_n(\cdot;\xi^\ast)}\Big[ \int^t_0L(\eta_\Delta(\gamma)(s),s,\eta_{\Delta}(\gamma)'(s)  )ds+v^0(\eta_\Delta(\gamma)(0) ) \\
&&- \Big(  \int^t_0L(\gamma^\ast(s),s,\gamma^\ast{}'(s)  )ds+v^0(\gamma^\ast(0) ) \Big)
 \Big]+h(t_{l+1}-t) +O(\sqrt{\dx}). 
\end{eqnarray*}
 Hence, we obtain 
 \begin{eqnarray}\label{sharp}
&& E_{\mu^{l+1,0}_n(\cdot;\xi^\ast)}\Big[ \int^t_0L(\eta_\Delta(\gamma)(s),s,\eta_{\Delta}(\gamma)'(s)  )ds+v^0(\eta_\Delta(\gamma)(0) ) \\\nonumber
&&\qquad\qquad - \Big(  \int^t_0L(\gamma^\ast(s),s,\gamma^\ast{}'(s)  )ds+v^0(\gamma^\ast(0) ) \Big)  \Big]=O(\sqrt{\dx}).
 \end{eqnarray}
Consider the set 
 \begin{eqnarray*}
 \Omega^+&:=&\Big\{\gamma\in\Omega^{l+1,0}_n \,|\,  \int^t_0L(\eta_\Delta(\gamma)(s),s,\eta_{\Delta}(\gamma)'(s)  )ds+v^0(\eta_\Delta(\gamma)(0) )  \\
 &&\qquad\qquad\qquad- \Big(  \int^t_0L(\gamma^\ast(s),s,\gamma^\ast{}'(s)  )ds+v^0(\gamma^\ast(0) ) \Big)\ge \dx^{\frac{1}{4}}  \Big\}.
 \end{eqnarray*}
 Since $\gamma^\ast$ is a minimizing curve for $v(x,t)$, we have for each $\gamma\in \Omega_n^{l+1,0}$,
 \begin{eqnarray*}
 0&\le& \int^t_0L(\eta_\Delta(\gamma)(s)+x-x_n,s,\eta_{\Delta}(\gamma)'(s)  )ds+v^0(\eta_\Delta(\gamma)(0) +x-x_n) \\
&&- \Big(  \int^t_0L(\gamma^\ast(s),s,\gamma^\ast{}'(s)  )ds+v^0(\gamma^\ast(0) )\Big)\\
&\le& \int^t_0L(\eta_\Delta(\gamma)(s),s,\eta_{\Delta}(\gamma)'(s)  )ds+v^0(\eta_\Delta(\gamma)(0) ) \\
&&- \Big(  \int^t_0L(\gamma^\ast(s),s,\gamma^\ast{}'(s)  )ds+v^0(\gamma^\ast(0) ) \Big) +\beta_7\dx.
 \end{eqnarray*}
Hence, noting that $\sum_{\Omega_n^{l+1,0}}=\sum_{\Omega^+}+\sum_{\Omega_n^{l+1,0}\setminus \Omega^+}$ in (\ref{sharp}), we obtain 
$$O(\sqrt{\dx})\ge  prob(\Omega^+)\dx^{\frac{1}{4}}+prob(\Omega_n^{l+1,0}\setminus \Omega^+)(-\beta_7\dx)\ge prob(\Omega^+)\dx^{\frac{1}{4}}-\beta_7\dx,$$
 which yields $prob(\Omega^+)=O(\dx^{\frac{1}{4}})$. 
Since $\gamma^\ast$ is the unique minimizing curve, it follows from Lemma \ref{cal} that $\Omega^{l+1,0}_n\setminus \Omega^+\subset \tilde{\Omega}^\varepsilon_\Delta$ for $\dx\ll \varepsilon$, which means that $prob(\tilde{\Omega}^\varepsilon_\Delta)\to 1$ as $\Delta\to0$. Therefore,  we obtain the convergence $E_{\mu^{l+1,0}_n(\cdot;\xi^\ast)}[ \norm\eta_\Delta(\gamma)-\gamma^\ast\norm_{L^2([0,t])}^2 ]\to0$ as $\Delta\to0$, which implies that the left hand side of (\ref{star}) tends to $0$ as $\Delta\to0$. From this, we see that for any $\varepsilon'>0$ there exists $\delta(\ep')>0$ such that if $|\Delta|<\delta(\ep')$ we have 
$$E_{\mu_n^{l+1,0}(\cdot;\xi^\ast)}[ \norm\gamma_\Delta-(\gamma^\ast+\gamma_\Delta(t)-x)\norm_{L^2([0,t])} ]\le \varepsilon'.$$
 Define $\Omega^{++}:=\{\gamma\in\Omega^{l+1,0}_n \,|\,   \norm\gamma_\Delta-(\gamma^\ast+\gamma_\Delta(t)-x)\norm_{L^2([0,t])} \ge\sqrt{\varepsilon'} \}$.  
 Then, we have $prob(\Omega^{++})\le\sqrt{\varepsilon'}$. By Lemma \ref{convert} (note that $\gamma_\Delta(t)-(\gamma^\ast(t)+\gamma_\Delta(t)-x)=0$), we obtain 
 $\norm \gamma_\Delta-(\gamma^\ast+\gamma_\Delta(t)-x)\norm_{C^0([0,t])}\le  O(\ep'{}^{\frac{1}{4}})$ for all $\gamma\in\Omega^{l+1,0}_n\setminus \Omega^{++}$. If $\ep' \ll \ep^4$, we have  $\norm \gamma_\Delta-\gamma^\ast\norm_{C^0([0,t])}\le \varepsilon$  for all $\gamma\in\Omega^{l+1,0}_n\setminus \Omega^{++}$ with $|\Delta|< \delta(\ep')$.  Therefore, it holds that, for any $\ep'>0$ with $\ep'\ll \ep$, there exists $\delta(\ep')>0$ such that if  $|\Delta|< \delta(\ep')$ we have  $\Omega^{l+1,0}_n\setminus \Omega^{++}\subset\Omega^\varepsilon_\Delta$ and $prob(\Omega^\varepsilon_\Delta)\ge 1- \sqrt{\ep'}$, which means $prob(\Omega^\varepsilon_\Delta)\to1$ as $\Delta\to0$.

Note that, for each fixed $m$, the same convergence result holds for the minimizing random walks for $v^{l+1}_{n+m}$, where $n$ is taken so that $x\in[x^1_{n^1-1},x^1_{n^1+1})\times\cdots\times[x^d_{n^d-1},x^d_{n^d+1})$ accordingly to $\Delta$.    

\indent We prove (3). Let $(x,t)$ be an arbitrary regular point with $t>0$. Let $\gamma^\ast$ be the unique minimizing curve for $v(x,t)$. We have for $j=1,\cdots,d$, 
\begin{eqnarray}\label{555555}
v_{x^j}(x,t)=\int^t_0 L_{x^j}(\gamma^\ast(s),s,\gamma^\ast{}'(s))ds+v^0_{x^j}(\gamma^\ast(0)).
\end{eqnarray}
 Let $(x_{n+1},t_{l+1})\in \G$ be a point such that $t\in [t_{l+1},t_{l+2})$ and  $x\in[x^1_{n^1-1},x^1_{n^1+1})\times\cdots\times[x^d_{n^d-1},x^d_{n^d+1})$.  Then, $u^j_\Delta(x,t)=(D_{x^j}v)^{l+1}_{n+\1}$. By (4) of Theorem \ref{main1}, we have 
\begin{eqnarray*}  
(D_{x^j}v)^{l+1}_{n+\1}&\le& E_{\mu_{n+\1-e_j}^{l+1,0}(\cdot;\xi^\ast(-e_j))}\Big[  \sum_{0<k\le l+1} L_{x^j}(\gamma^k,t_{k-1},\xi^k_{m(\gamma^{k})}(-e_j))\dt+ (D_{x^j}v)^{0}_{m(\gamma^0)+e_j}\Big]\\
&&+\theta \dx\\
&\le&E_{\mu_{n+\1-e_j}^{l+1,0}(\cdot;\xi^\ast(-e_j))}\Big[ \int_0^t L_{x^j}(\eta_\Delta(\gamma)(s),s,\eta_\Delta(\gamma)'(s))ds+ (D_{x^j}v)^{0}_{m(\gamma^0)+e_j}\Big]\\
&&+\beta_8 \sqrt{\dx}.
\end{eqnarray*}
Hence, we have 
\begin{eqnarray*}
&&(D_{x^j}v)^{l+1}_{n+\1}-v_{x^j}(x,t)\le R_1+R_2+\beta_8 \sqrt{\dx},\\
&&R_1:=E_{\mu_{n+\1-e_j}^{l+1,0}(\cdot;\xi^\ast(-e_j))}\Big[ \int_0^t \big(L_{x^j}(\eta_\Delta(\gamma)(s),s,\eta_\Delta(\gamma)'(s)) - L_{x^j}(\gamma^\ast(s),s,\gamma^\ast{}'(s))\big)ds\Big], \\
&&R_2:=  E_{\mu_{n+\1-e_j}^{l+1,0}(\cdot;\xi^\ast(-e_j))}\Big[(D_{x^j}v)^{0}_{m(\gamma^0)+e_j}-v^0_{x^j}(\gamma^\ast(0))\Big]. 
\end{eqnarray*}
We already know from the proof of (2) that, for each $j$, we have $\norm \eta_\Delta(\gamma)-\gamma^\ast\norm_{C^0([0,t])}\to0$,  $\norm \eta_\Delta(\gamma)'-\gamma^\ast{}'\norm_{L^2([0,t])}\to0$ in probability as $\Delta\to0$, which yields $R_1\to0$ as $\Delta\to0$.  Since $v^0$ is semiconcave, we have 
$$\sup_{\{x: \norm x-\gamma^\ast(0)\norm_{\infty}\le \varepsilon\}}|v^0_{x^j}(x)-v^0_{x^j}(\gamma^\ast(0))|\to0\mbox{ \,\, as $\varepsilon\to0$}$$
 (see, e.g.,  Proposition 3.3.4 in \cite{Cannarsa}). By (2), we have $\gamma^0\to\gamma^\ast(0)$ in probability as $\Delta\to0$. Therefore, we see that 
\begin{eqnarray*}
(D_{x^j}v)^{0}_{m(\gamma^0)+e_j}&=&
\frac{v^0(\gamma^0+2\dx\cdot e_j)-v^0(\gamma^0)}{2\dx}
= \frac{1}{2\dx}\int_0^{2\dx} v^0_{x^j}(\gamma^{0}+ye_j)dy\\
&&\to v^0_{x^j}(\gamma^\ast(0)) \mbox{\,\,\, in probability as $\Delta\to0$} 
\end{eqnarray*}
in the case of (\ref{initial-1}), and that 
\begin{eqnarray*}
(D_{x^j}v)^{0}_{m(\gamma^0)+e_j}&=&\frac{1}{(2\dx)^{d}} \int_{[-\dx,\dx]^d}  \frac{v^0(\gamma^0+2\dx\cdot e_j+y)-v^0(\gamma^0+y)}{2\dx} dy\\
&=& \frac{1}{(2\dx)^{d}} \int_{[-\dx,\dx]^d}\Big(\frac{1}{2\dx}\int_0^{2\dx} v^0_{x^j}(\gamma^{0}+se_j+y)ds\Big)dy\\
&&\to v^0_{x^j}(\gamma^\ast(0)) \mbox{\,\,\, in probability as $\Delta\to0$} 
\end{eqnarray*}
in the case of (\ref{initial-2}), which yields $R_2\to0$ as $\Delta\to0$. Thus, we conclude that
$$\limsup_{\Delta\to0}(u^j_\Delta(x,t)-v_{x^j}(x,t))\le0.$$
Similar reasoning yields  
$$\liminf_{\Delta\to0}(u^j_\Delta(x,t)-v_{x^j}(x,t))\ge0.$$
\indent We prove (4). Let $v^{0\delta}$ be the mollified function of $v^0$ with the standard mollifier. Let $v^\delta$, $v^\delta{}^k_{m+\1}$ be the solution of (\ref{HJ}), (\ref{HJ-Delta})  with the initial data $v^{0\delta}$, respectively.  The interpolation of $(D_{x^j}v^\delta)^k_m$  is denoted by $u^\delta{}^j_\Delta $.  We use the notation $u^j{}^k_{m}:=(D_{x^j}v)^k_m$, $u^{j\delta}{}^k_{m}:=(D_{x^j}v^\delta)^k_m$. 

We start with 
\begin{eqnarray*}
\norm u^j_\Delta(\cdot,t)-v_{x^j}(\cdot,t)\norm_{L^1(K)}&\le&
\norm u^j_\Delta(\cdot,t)-u^\delta{}^j_\Delta(\cdot,t)\norm_{L^1(K)}\\
&&+\norm u^\delta{}^j_\Delta(\cdot,t)-v^\delta_{x^j}(\cdot,t)\norm_{L^1(K)}\\
&&+\norm v^\delta_{x^j}(\cdot,t)-v_{x^j}(\cdot,t)\norm_{L^1(K)}. 
\end{eqnarray*}
Take any sequence $\delta_i\to0+$ as $i\to\infty$ and set $\delta=\delta_i$. Let  $(x,t)$ be a point such that $v_x(x,t)$ and $v^{\delta_i}_x(x,t)$ ($i\in\N$) exist  (a.e. points are such). The minimizing curve $\gamma^\ast$ of $v(x,t)$ and $\gamma^\ast{}^{\delta_i}$ of $v^{\delta_i}(x,t)$ with any $i\in\N$ lie on a compact set $\tilde{K}$ due to Lemma \ref{lemma-regularity}. The variational representation of  $v(x,t)$ and $v^{\delta_i}(x,t)$ implies $|v^{\delta_i}(x,t)-v(x,t)|\le \sup_{y\in\tilde{K}}|v^0{}^{\delta_i}(y)-v^0(y)|\to0$ as $i\to0$. Hence, it follows from Lemma \ref{cal} that $\gamma^\ast{}^{\delta_i}\to\gamma^\ast$,  $\gamma^\ast{}^{\delta_i}{}'\to\gamma^\ast{}'$ uniformly as $i\to\infty$ (note that $\gamma^\ast{}^{\delta_i},\gamma^\ast$ are $C^2$-solutions of the Euler-Lagrange equation).
Therefore, (3) of Lemma \ref{lemma-regularity} implies that $v^{\delta_i}_{x^j}(x,t)-v_{x^j}(x,t)\to0$ as $i\to\infty$.  Thus, we conclude that  $\norm v^\delta_{x^j}(\cdot,t)-v_{x^j}(\cdot,t)\norm_{L^1(K)}  \to0$ as $\delta\to0+$.  

Let $\Gamma\subset\R^d$ be a closed cube containing $K$ and let $\Gamma+a$ with $a>0$ denote the $a$-neighborhood of $\Gamma$.   Fix  $\ep>0$ arbitrarily. Fix $\delta>0$ so that 
$$\norm v^\delta_{x^j}(\cdot,t)-v_{x^j}(\cdot,t)\norm_{L^1(K)}<\frac{\ep}{3},\quad \sum_{j=1}^d\norm v_{x^j}^{0\delta}-v^0_{x^j}\norm_{L^1(\Gamma+T/\lambda_0+1)}
< \frac{\ep}{3}.$$
Since $v^{0\delta}$ is semiconcave, (3) implies that there exists $\alpha>0$ such that if $|\Delta|<\alpha$ we have  $\norm u^\delta{}^j_\Delta(\cdot,t)-v^\delta_{x^j}(\cdot,t)\norm_{L^1(K)}<\ep/3$.  
Now we estimate $\norm u^j_\Delta(\cdot,t)-u^\delta{}^j_\Delta(\cdot,t)\norm_{L^1(K)}$. Let $\Gamma^{l+1}$ be a set  of all indexes $m+\1$ such that $x_{m+\1}\subset\Gamma$ and $(x_{m+\1},t_{l+1})\in \G$. Define $\Gamma^{l}:=\{ m+\1+\omega\,|\, x_{m+\1}\in\Gamma^{l+1},\omega\in B\}, \Gamma^{l-1}, \ldots,\Gamma^0$ in a recurrent way.  
\begin{Prop}\label{L1-contraction}
Let $v^k_{m+\1}$, $\tilde{v}^k_{m+\1}$ satisfy the difference equation in (\ref{HJ-Delta}) with $\lambda<\lambda_1$, where $\lambda_1$ is mentioned in Theorem \ref{main1}. Set $u^k_m=(u^1{}^k_m,\ldots,u^d{}^k_m)$, $u^j{}^k_m:=(D_{x^j}v)^k_{m}$, $\tilde{u}^k_m=(\tilde{u}^1{}^k_m,\ldots,\tilde{u}^d{}^k_m)$, $\tilde{u}^j{}^k_m:=(D_{x^j}\tilde{v})^k_{m}$. Then,  we have 
 $$\sum_{m+\1\in\Gamma^{l+1}} \sum_{j=1}^d   |\tilde{u}^j{}^{l+1}_{m+\1}-u^j{}^{l+1}_{m+\1}|\le \sum_{m\in\Gamma^{l}}  \sum_{j=1}^d   |\tilde{u}^j{}^{l}_{m}-u^j{}^{l}_{m}|.$$ 
\end{Prop}
\begin{proof}
By the difference equation in (\ref{HJ-Delta}), we have 
\begin{eqnarray*}  
 v^{l+1}_{m+\1+e_j}&=&\frac{1}{2d}\sum_{i=1}^d (v^l_{m+\1+e_j+e_i}+v^l_{m+\1+e_j-e_i})-H(x_{m+\1+e_j},t_l,u^l_{m+\1+e_j})\dt,\\
 v^{l+1}_{m+\1-e_j}&=&\frac{1}{2d}\sum_{i=1}^d (v^l_{m+\1-e_j+e_i}+v^l_{m+\1-e_j-e_i})-H(x_{m+\1-e_j},t_l,u^l_{m+\1-e_j})\dt, 
\end{eqnarray*} 
which yields 
\begin{eqnarray*}
u^j{}^{l+1}_{m+\1}&=&\frac{1}{2d}\sum_{i=1}^d (u^j{}^l_{m+\1+e_i}+u^j{}^l_{m+\1-e_i})\\
&&-\frac{\lambda}{2}(H(x_{m+\1+e_j},t_l,u^l_{m+\1+e_j})-H(x_{m+\1-e_j},t_l,u^l_{m+\1-e_j})).
\end{eqnarray*}
Hence, we obtain 
\begin{eqnarray*}
\tilde{u}^j{}^{l+1}_{m+\1}-u^j{}^{l+1}_{m+\1}&=&\frac{1}{2d}\sum_{i=1}^d \{ (\tilde{u}^j{}^l_{m+\1+e_i}-u^j{}^l_{m+\1+e_i})+(\tilde{u}^j{}^l_{m+\1-e_i}-u^j{}^l_{m+\1-e_i}) \}\\
&&-\frac{\lambda}{2}\Big(H(x_{m+\1+e_j},t_l,\tilde{u}^l_{m+\1+e_j})-H(x_{m+\1+e_j},t_l,u^l_{m+\1+e_j})\\
&&-H(x_{m+\1-e_j},t_l,\tilde{u}^l_{m+\1-e_j})+H(x_{m+\1-e_j},t_l,u^l_{m+\1-e_j})\Big)\\
&=&\frac{1}{2d}\sum_{i=1}^d \{ (\tilde{u}^j{}^l_{m+\1+e_i}-u^j{}^l_{m+\1+e_i})+(\tilde{u}^j{}^l_{m+\1-e_i}-u^j{}^l_{m+\1-e_i}) \}\\
&&-\frac{\lambda}{2}\sum_{i=1}^d\{ \zeta^i{}^l_{m+\1+e_j} (\tilde{u}^i{}^l_{m+\1+e_j}-u^i{}^l_{m+\1+e_j})\\
&&-\zeta^i{}^l_{m+\1-e_j}(\tilde{u}^l_{m+\1-e_j}-u^l_{m+\1-e_j})\},
\end{eqnarray*} 
where $ \zeta^i{}^l_{m+\1\pm e_j}:=H_{p^i}(x_{m+\1\pm e_j},t_l,u^l_{m+\1\pm e_j}+\theta^l_{m+\1\pm e_j} (\tilde{u}^l_{m+\1\pm e_j}-u^l_{m+\1\pm e_j}))$ with $\theta^l_{m+\1\pm e_j}\in(0,1)$ coming from the Taylor expansion. Let $\sigma^j{}^{l+1}_{m+\1}=\pm1$ denote the sign of  $\tilde{u}^j{}^{l+1}_{m+\1}-u^j{}^{l+1}_{m+\1}$. Then, we obtain 
\begin{eqnarray*}
\sum_{j=1}^d|\tilde{u}^j{}^{l+1}_{m+\1}-u^j{}^{l+1}_{m+\1}|&=&\sum_{i,j=1}^d \Big( \frac{1}{2d}\sigma^i{}^{l+1}_{m+\1} -\frac{\lambda}{2} \zeta^i{}^{l}_{m+\1+e_j}\sigma^j{}^{l+1}_{m+\1}\Big)\big(\tilde{u}^i{}^l_{m+\1+e_j}-u^i{}^l_{m+\1+e_j}\big)\\
&+&\sum_{i,j=1}^d \Big( \frac{1}{2d}\sigma^i{}^{l+1}_{m+\1} +\frac{\lambda}{2} \zeta^i{}^{l}_{m+\1-e_j}\sigma^j{}^{l+1}_{m+\1}\Big)\big(\tilde{u}^i{}^l_{m+\1-e_j}-u^i{}^l_{m+\1-e_j}\big).
\end{eqnarray*} 
Define $\partial\Gamma^l:=\{ m\in\Gamma^l \,|\, \{ m\pm e_j\,|\,j=1,\ldots,d \}\not\subset \Gamma^{l+1} \}$. Noting cancelation, we obtain 
\begin{eqnarray*}
\sum_{m+\1\in\Gamma^{l+1}}\sum_{j=1}^d|\tilde{u}^j{}^{l+1}_{m+\1}-u^j{}^{l+1}_{m+\1}|&=&\sum_{m\in \Gamma^l\setminus\partial\Gamma^l}\sum_{i=1}^d \Big\{\sum_{j=1}^d \frac{1}{2d}\Big(\sigma^i{}^{l+1}_{m+e_j}+\sigma^i{}^{l+1}_{m-e_j}\Big)\Big\}(\tilde{u}^i{}^l_{m}-u^i{}^l_{m})\\
&+&\sum_{m\in  \partial\Gamma^l}\sum_{i=1}^d \Big( \sum_{j=1}^d \delta^{ij}_m\Big)(\tilde{u}^i{}^l_{m}-u^i{}^l_{m}),
\end{eqnarray*} 
where 
\begin{eqnarray*}
\delta^{ij}_m:=\left\{
\begin{array}{lll}
&\dis \frac{1}{2d}\Big(\sigma^i{}^{l+1}_{m+e_j}+\sigma^i{}^{l+1}_{m-e_j}\Big), \mbox{\qquad  if $m+e_j\in\Gamma^{l+1}$, $m- e_j\in\Gamma^{l+1}$},\medskip \\
&\dis \frac{1}{2d}\sigma^i{}^{l+1}_{m-e_j}-\frac{\lambda}{2}\zeta^i{}^l_{m}\sigma^j{}^{l+1}_{m-e_j}, \mbox{ if $m+ e_j\not\in\Gamma^{l+1}$, $m- e_j\in\Gamma^{l+1}$},\medskip\\
&\dis \frac{1}{2d}\sigma^i{}^{l+1}_{m+e_j}+\frac{\lambda}{2}\zeta^i{}^l_{m}\sigma^j{}^{l+1}_{m+e_j}, \mbox{ if $m+ e_j\in\Gamma^{l+1}$, $m- e_j\not\in\Gamma^{l+1}$},\medskip\\
& 0,  \mbox{ \qquad\qquad\qquad\qquad\qquad if $m+ e_j\not\in\Gamma^{l+1}$, $m- e_j\not\in\Gamma^{l+1}$}.
\end{array}
\right.
\end{eqnarray*}
 Due to the choice $\lambda< \lambda_1$, we have $|\delta^{ij}_m|\le1/d$ and 
 $$ \Big|\sum_{j=1}^d \frac{1}{2d}\Big(\sigma^i{}^{l+1}_{m+e_j}+\sigma^i{}^{l+1}_{m-e_j}\Big)\Big|\le1,\quad \Big| \sum_{j=1}^d \delta^{ij}_m\Big|\le 1,$$
 which completes the proof.
\end{proof}
\noindent Note that hyperbolic scaling implies $\{x_m \,|\,m\in\Gamma^0 \}\subset \Gamma+T/\lambda_0$. By Proposition  \ref{L1-contraction} with $\tilde{v}^k_{m+\1}=v^\delta{}^k_{m+\1}$, $\tilde{u}^j{}^k_m=u^{j\delta}{}^k_{m}$, we obtain 
\begin{eqnarray} \nonumber
&&\norm u^j_\Delta(\cdot,t)-u^\delta{}^j_\Delta(\cdot,t)\norm_{L^1(K)}
\le \sum_{m+\1\in\Gamma^{l+1}} \sum_{j=1}^d   |\tilde{u}^j{}^{l+1}_{m+\1}-u^j{}^{l+1}_{m+\1}|(2\dx)^d\\\nonumber
&&\le \sum_{m\in\Gamma^0} \sum_{j=1}^d   |\tilde{u}^j{}^{0}_{m}-u^j{}^{0}_{m}|(2\dx)^d\\\label{313}
&&= \sum_{m\in\Gamma^0} \sum_{j=1}^d   \Big |
 \frac{1}{(2\dx)^{d}} \int_{[-\dx,\dx]^d}\Big(\frac{1}{2\dx}\int_{-\dx}^{\dx} v_{x^j}^{0\delta}(x_m+se_j+y) \\\nonumber
 &&\qquad\qquad\qquad\qquad\qquad\qquad  -v^0_{x^j}(x_m+se_j+y)ds\Big)dy\Big|(2\dx)^d\\\nonumber
 &&\le \sum_{j=1}^d\norm v_{x^j}^{0\delta}-v^0_{x^j}\norm_{L^1(\Gamma+T/\lambda_0+1)}<\frac{\ep}{3}.
\end{eqnarray}
This concludes the proof.
\end{proof} 
The reason why we need (\ref{initial-2}) is to obtain the $L^1$-norm of the measurable function $v_{x^j}^{0\delta}-v^0_{x^j}$ from (\ref{313}). If $d=1$, it is clear that we may follow the above proof of (4) with \eqref{initial-1}.   

We conclude this section with discussion on Remark \ref{remark22}.  Currently, we fail to obtain (3) of Theorem \ref{main2} without semiconcavity of initial data, if $d\ge2$. We will see the difficulty and how to overcome it for $d=1$. Without semiconcavity, uncountably many minimizing curves $\gamma^\ast(s)$ meet at one point at $s=0$ (so-called ``rarefaction'') and $v^0_x(\gamma^\ast(0))$ does not exist for such $\gamma^\ast$. In such a case, we may use the following formula: For any $\tau\in[0,t)$,  
\begin{eqnarray}\label{555555}
v_{x^j}(x,t)=\int^t_\tau L_{x^j}(\gamma^\ast(s),s,\gamma^\ast{}'(s))ds+L_{\xi^j}(\gamma^\ast(\tau),\tau,\gamma^\ast{}'(\tau)).
\end{eqnarray}
We have a discrete version of \eqref{555555}: We observe that for each $e_j$,
\begin{eqnarray*}
v^{l+1}_{n+\1-e_j}&=&\sum_{\gamma\in\Omega^{l+1,0}_{n+\1-e_j}}\mu^{l+1,0}_{n+\1-e_j}(\gamma;\xi^\ast|_{G^{l+1,1}_{n+\1-e_j}})\Big( \sum_{0<k\le l+1}L(\gamma^k,t_{k-1},\xi^\ast{}^k_{m(\gamma^k)})\dt\\
&&+v^0(\gamma^0) \Big)+ ht_{l+1}\\
&=&\sum_{\gamma\in\Omega^{l+1,k(\tau)}_{n+\1-e_j}}\mu^{l+1,k(\tau)}_{n+\1-e_j}(\gamma;\xi^\ast|_{G^{l+1,k(\tau)+1}_{n+\1-e_j}})\Big( \sum_{k(\tau)<k\le l+1}L(\gamma^k,t_{k-1},\xi^\ast{}^k_{m(\gamma^k)})\dt\\
&&+v^{k(\tau)}_{m(\gamma^{k(\tau)})} \Big)+ h(t_{l+1}-t_{k(\tau)}).
\end{eqnarray*}
For each $e_j$, define the control $\zeta$ on $G^{l+1,1}_{n+\1+e_j}$ as
\begin{eqnarray*}
\zeta(x_m,t_{k+1}):=\left\{
\begin{array}{lll}
& \xi^\ast(x_m-e_j\cdot2\dx,t_{k+1})   \mbox{ \,\,\,for $k(\tau)<k+1\le l+1$,} \\
& \xi^\ast(x_m,t_{k+1})   \mbox{ \,\,\,for $0<k+1\le k(\tau)$.} 
\end{array}
\right.
\end{eqnarray*}
Then, we have 
\begin{eqnarray*}
v^{l+1}_{n+\1+e_j}&\le& \sum_{\gamma\in\Omega^{l+1,0}_{n+\1+e_j}}\mu^{l+1,0}_{n+\1+e_j}(\gamma;\zeta)\Big( \sum_{0<k\le l+1}L(\gamma^k,t_{k-1},\zeta^{k}_{m(\gamma^k)})\dt +v^0(\gamma^k)  \Big) + ht_{l+1}\\
&=&\sum_{\gamma\in\Omega^{l+1,k(\tau)}_{n+\1-e_j}}\mu^{l+1,k(\tau)}_{n+\1-e_j}(\gamma;\xi^\ast|_{G^{l+1,k(\tau)+1}_{n+\1-e_j}})\\
&&\times\Big( \sum_{k(\tau)<k\le l+1}L(\gamma^k+e_j\cdot2\dx,t_{k-1},\xi^\ast{}^{k}_{m(\gamma^k)})\dt
 +v^{k(\tau)}_{m(\gamma^{k(\tau)}+e_j\cdot2\dx)}  \Big)\\
 &&+ h(t_{l+1}-t_{k(\tau)}).
\end{eqnarray*}
Hence, with Lemma \ref{limit theorem}, we obtain 
\begin{eqnarray*}
(D_{x^j}v)^{l+1}_{n+\1}&=&\frac{v^{l+1}_{n+\1+e_j}-v^{l+1}_{n+\1-e_j}}{2\dx}\\
&\le& \sum_{\gamma\in\Omega^{l+1,k(\tau)}_{n+\1-e_j}}\mu^{l+1,k(\tau)}_{n+\1-e_j}(\gamma;\xi^\ast|_{G^{l+1,k(\tau)+1}_{n+\1-e_j}})\\
&&\times\Big( \sum_{k(\tau)<k\le l+1}L_{x^j}(\gamma^k,t_{k-1},\xi^\ast{}^{k}_{m(\gamma^k)})\dt+ (D_{x^j}v)^{k(\tau)}_{m(\gamma^{k(\tau)}+e_j\dx)}  \Big)\\
&&+\beta_8\dx\\
&\le&  \sum_{\gamma\in\Omega^{l+1,k(\tau)}_{n+\1-e_j}}\mu^{l+1,k(\tau)}_{n+\1-e_j}(\gamma;\xi^\ast|_{G^{l+1,k(\tau)+1}_{n+\1-e_j}})\\
&&\times\Big(\int^t_\tau L_{x^j}(\eta_\Delta(\gamma)(s),s,\eta_\Delta(\gamma)'(s))ds+ (D_{x^j}v)^{k(\tau)}_{m(\gamma^{k(\tau)}+e_j\dx)}  \Big)\\
&&+\beta_9\sqrt{\dx}.
\end{eqnarray*}
Now, we compare this inequality and (\ref{555555}). 
Since we already know about the convergence of the minimizing random walks $\gamma$ and $\eta(\gamma)$ to $\gamma^\ast$, it is enough to estimate 
\begin{eqnarray}\label{zure}
&& \sum_{\gamma\in\Omega^{l+1,k(\tau)}_{n+\1-e_j}}\mu^{l+1,k(\tau)}_{n+\1-e_j}(\gamma;\xi^\ast|_{G^{l+1,k(\tau)+1}_{n+\1-e_j}})\Big((D_{x^j}v)^{k(\tau)}_{m(\gamma^{k(\tau)}+e_j\dx)} \Big)-L_{\xi^j}(\gamma^\ast(\tau),\tau,\gamma^\ast{}'(\tau))\\\nonumber
&&\quad =\sum_{\gamma\in\Omega^{l+1,k(\tau)}_{n+\1-e_j}}\mu^{l+1,k(\tau)}_{n+\1-e_j}(\gamma;\xi^\ast|_{G^{l+1,k(\tau)+1}_{n+\1-e_j}})\\\nonumber
&&\quad\quad\quad\quad\quad\quad\times\Big(L_{\xi^j}(\gamma^{k(\tau)}+e_j\dx,t_{k(\tau)},{\xi^\ast}_{m(\gamma^{k(\tau)}+e_j\dx)}^{k(\tau)+1})-L_{\xi^j}(\gamma^\ast(\tau),\tau,\gamma^\ast{}'(\tau))\Big), 
\end{eqnarray}
where we note that $(D_{x^j}v)^k_m=L_{\xi^j}(x_m,t_k,{\xi^\ast}_m^{k+1})$. Unfortunately, ${\xi^\ast}_{m(\gamma^{k(\tau)}+e_j\dx)}^{k(\tau)+1}$ is NOT equal to ${\xi^\ast}_{m(\gamma^{k(\tau)+1})}^{k(\tau)+1}$, i.e., $\gamma^{k(\tau)+1}$ is not necessarily equal to $\gamma^{k(\tau)}+e_j\dx$, which prevents us from direct application of the $L^2$-convergence of $\eta_\Delta(\gamma)'(\cdot)$ to $\gamma^\ast{}'(\cdot)$. 

If $d=1$, we manage to estimate (\ref{zure}) by means of the one-sided Lipschitz estimate of the discrete derivative, or the entropy condition, proved in Proposition 2.8 of \cite{Soga3} (this holds also in our whole space setting): There exists $M^k>0$ independent of $\Delta$ such that 
$$\frac{(D_{x}v)^k_{m+2}-(D_{x}v)^k_m}{2\dx}\le M^k.$$   
Let $\Omega_x$ be the set of all $\gamma \in \Omega^{l+1,k(\tau)+1}_{n}$ such that $\gamma^{k(\tau)+1}=x$.  We have
\begin{eqnarray*}
 &&\sum_{\gamma\in\Omega^{l+1,k(\tau)}_{n}}\mu^{l+1,k(\tau)}_{n}(\gamma;\xi^\ast|_{G^{l+1,k(\tau)+1}_{n}})\Big((D_{x}v)^{k(\tau)}_{m(\gamma^{k(\tau)}+\dx)}  - (D_{x}v)^{k(\tau)}_{m(\gamma^{k(\tau)+1})}\Big) \\
 &&=\sum_{x\in X^{l+1,k(\tau)+1}_{n}}\sum_{\gamma\in\Omega_x}\mu^{l+1,k(\tau)+1}_{n}(\gamma;\xi^\ast|_{G^{l+1,k(\tau)+2}_{n}}) \Big\{
 \rho^{k(\tau)+1}_{m(x)} (+1) \Big((D_{x}v)^{k(\tau)}_{m(x+\dx+\dx)} \\
&& \quad - (D_{x}v)^{k(\tau)}_{m(x)}\Big)
 +  \rho^{k(\tau)+1}_{m(x)} (-1) \Big((D_{x}v)^{k(\tau)}_{m(x-\dx+\dx)}  - (D_{x}v)^{k(\tau)}_{m(x)}\Big) \Big\}  \\
 &&\le M^{k(\tau)}\cdot 2\dx.
\end{eqnarray*}
Hence, we obtain 
\begin{eqnarray*}
\eqref{zure}|_{d=1}&\le&  \sum_{\gamma\in\Omega^{l+1,k(\tau)}_{n}}\mu^{l+1,k(\tau)}_{n}(\gamma;\xi^\ast|_{G^{l+1,k(\tau)+1}_{n}})(D_{x^j}v)^{k(\tau)}_{m(\gamma^{k(\tau)+1})} +M^{k(\tau)}\cdot 2\dx\\
&&-L_{\xi}(\gamma^\ast(\tau),\tau,\gamma^\ast{}'(\tau))\\
&=&   \sum_{\gamma\in\Omega^{l+1,k(\tau)}_{n}}\mu^{l+1,k(\tau)}_{n}(\gamma;\xi^\ast|_{G^{l+1,k(\tau)+1}_{n}})\Big(L_\xi(\gamma^{k(\tau)+1},t_{k(\tau)},\xi^\ast{}^{k(\tau)+1}_{m(\gamma^{k(\tau)+1})})\\
&&-L_{\xi}(\gamma^\ast(\tau),\tau,\gamma^\ast{}'(\tau))\Big) +M^{k(\tau)}\cdot 2\dx\\
&=&   \sum_{\gamma\in\Omega^{l+1,0}_{n}}\mu^{l+1,0}_{n}(\gamma;\xi^\ast|_{G^{l+1,1}_{n}})\Big(L_\xi(\gamma^{k(\tau)+1},t_{k(\tau)},\xi^\ast{}^{k(\tau)+1}_{m(\gamma^{k(\tau)+1})})\\
&&-L_{\xi}(\gamma^\ast(\tau),\tau,\gamma^\ast{}'(\tau))\Big) +M^{k(\tau)}\cdot 2\dx.
\end{eqnarray*}
According to the above proof of (2) of Theorem \ref{main2},  for any $\ep>0$, there exists $\delta>0$ such that if $|\Delta|< \delta$ we have 
\begin{eqnarray*}
E_{\mu^{l+1,0}_{n}(\cdot;\xi^\ast|_{G^{l+1,1}_{n}})}[\norm \eta_\Delta(\gamma){}'-\gamma^\ast{}'\norm_{L^1([0,t])}]&\le& E_{\mu^{l+1,0}_{n}(\cdot;\xi^\ast|_{G^{l+1,1}_{n}})}[\sqrt{t}\norm \eta_\Delta(\gamma){}'-\gamma^\ast{}'\norm_{L^2([0,t])}]\\
&<& \ep.
\end{eqnarray*}
 Therefore, there exists $\tau\in[t/2,t]$ such that 
 \begin{eqnarray*}
 E_{\mu^{l+1,0}_{n}(\cdot;\xi^\ast|_{G^{l+1,1}_{n}})}[| \eta_\Delta(\gamma){}'(\tau)-\gamma^\ast{}'(\tau)|]&=&E_{\mu^{l+1,0}_{n}(\cdot;\xi^\ast|_{G^{l+1,1}_{n}})}[| \xi^\ast{}^{k(\tau)+1}_{m(\gamma^{k(\tau)+1})}-\gamma^\ast{}'(\tau)|]\\
 &<&2\ep/t.
 \end{eqnarray*}
Since $M^{k(\tau)}$ with $\tau\ge t/2$ is bounded for $\Delta\to0$, we conclude that $\limsup_{\Delta\to0} \eqref{zure}|_{d=1}=0$ and $\limsup_{\Delta\to0}(u_\Delta (x,t)-v_x(x,t))=0$. In a similar way, we obtain  $\liminf_{\Delta\to0}(u_\Delta (x,t)-v_x(x,t))=0$.
Here is summary of remarks on the case $d=1$: All the results hold with (\ref{initial-1}); If the entropy condition is known,  a.e. pointwise convergence of the derivative can be proved  for any  initial data; Even if the entropy condition is not known, the $L^1$-convergence of the derivative can be proved for any initial data with (\ref{initial-1}).


\end{document}